\documentclass[hypertex,11pt]{article}
\usepackage{fullpage}
\usepackage{amsfonts,amsmath,amscd,amssymb,amsthm}
\usepackage[backrefs, alphabetic,initials]{amsrefs}
\usepackage{mathrsfs}
\usepackage[all]{xy}
\usepackage{amsmath, enumerate}
\usepackage{amsxtra}
\usepackage {pstricks}
\usepackage{pstricks,pst-node}
 
\def\convf{\hbox{\space \raise-2mm\hbox{$\textstyle      \bigotimes \atop \scriptstyle \omega$} \space}}
\def\0{{\bar 0}}
\def\1{{\bar 1}}
\def\C{{\mathbb C}}
\def\Z{{\mathbb Z}}

\def\N{{\mathbb N}}

\newcommand{\oa}{\bar{0}}

\def\Prim{{\operatorname{Prim}\;}}

\def\ad{{\operatorname{ad \;}}}

\def\rad{\operatorname{rad\;}}

\def\Ext{{\operatorname{Ext}}}
\def\Hom {{\operatorname{Hom}}}

\def\ann {{\operatorname{ann}}}

\newcommand{\ttl}{\mathtt{L}}

\newcommand{\itema}{\item[{{\rm(a)}}]}
\newcommand{\itemi}{\item[{{\rm(i)}}]}
\newcommand{\itemii}{\item[{{\rm(ii)}}]}
\newcommand{\itemiii}{\item[{{\rm(iii)}}]}

\newcommand{\itemb}{\item[{{\rm(b)}}]}

\newcommand{\itemc}{\item[{{\rm(c)}}]}

\newcommand{\itemo}{\item[{}]}
\newcommand{\noi}{\noindent}

\newcommand{\ey}{\end{eqnarray}}
\newcommand{\by}{\begin{eqnarray}}
\newcommand{\nn}{\nonumber}
\newcommand{\ga}{\alpha}
\newcommand{\gb}{\beta}
\newcommand{\gc}{\gamma}

\newcommand{\Gt}{\Theta}

\newcommand{\Gd}{\Delta}
\newcommand{\gd}{\delta}

\newcommand{\gt}{\tau}
\newcommand{\gz}{\zeta}
\newcommand{\gl}{\lambda}

\newcommand{\gr}{\rho}

\newcommand{\gep}{\epsilon}

\newcommand{\fg}{\mathfrak{g}}\newcommand{\fgl}{\mathfrak{gl}}
\newcommand{\fsl}{\mathfrak{sl}}

\newcommand{\fh}{\mathfrak{h}}

\newcommand{\fb}{\mathfrak{b}}

\newcommand{\fn}{\mathfrak{n}}

\newcommand{\fk}{\mathfrak{k}}

\newcommand{\sos}{\small{\mathfrak{osp}}}

\newcommand{\ff}{\footnote}
\newfont{\eufm}{eufm10 scaled\magstep1}
\newcommand{\pd}{\partial}

\newcommand{\cO}{\mathcal{O}}

\newcommand{\cI}{\mathcal{I}}

\newcommand{\cL}{\mathcal{L}}

\newcommand{\cH}{\mathcal{H}}

\newcommand{\bbZ}{\mathbb{Z}}

\newcommand{\bco}{\begin{conjecture}}
\newcommand{\ba}{\begin{alg}}
\newcommand{\ea}{\end{alg}}
\newcommand{\eco}{\end{conjecture}}
\newcommand{\bpf}{\begin{proof}}
\newcommand{\epf}{\end{proof}}
\newcommand{\bt}{\begin{theorem}}

\newcommand{\et}{\end{theorem}}

\newcommand{\br}{\begin{rem}}
\newcommand{\er}{\end{rem}}
\newcommand{\brs}{\begin{rems}}
\newcommand{\ers}{\end{rems}}
\newcommand{\bi}{\begin{itemize}}
\newcommand{\bl}{\begin{lemma}}
\newcommand{\bsul}{\begin{sublemma}}
\newcommand{\esul}{\end{sublemma}}
\newcommand{\bp}{\begin{proposition}}
\newcommand{\be}{\begin{equation}}
\newcommand{\bc}{\begin{corollary}}
\newcommand{\bexs}{\begin{examples}}
\newcommand{\eexs}{\end{examples}}
\newcommand{\bexa}{\begin{example}}
\newcommand{\eexa}{\end{example}}
\newcommand{\bex}{\begin{exercise}}
\newcommand{\eex}{\end{exercise}}
\newcommand{\btab}{\begin{tab}}
\newcommand{\etab}{\end{tab}}
\newcommand{\ei}{\end{itemize}}
\newcommand{\el}{\end{lemma}}
\newcommand{\ep}{\end{proposition}}
\newcommand{\ee}{\end{equation}}
\newcommand{\ec}{\end{corollary}}
\newcommand{\Bc}{\begin{center}}
\newcommand{\Ec}{\end{center}}
\newcommand{\bh}{\begin{hyp}}
\newcommand{\eh}{\end{hyp}}
\newcommand{\bhs}{\begin{hyps}}
\newcommand{\ehs}{\end{hyps}}
\newcommand{\bd}{\begin{dfn}}
\newcommand{\ed}{\end{dfn}}

\begin{document}
\title{Table of Contents}

\newtheorem{thm}{Theorem}[section]
\newtheorem{hyp}[thm]{Hypothesis}
 \newtheorem{hyps}[thm]{Hypotheses}
  \newtheorem{rems}[thm]{Remarks}

\newtheorem{conjecture}[thm]{Conjecture}
\newtheorem{theorem}[thm]{Theorem}
\newtheorem{theorem a}[thm]{Theorem A}
\newtheorem{example}[thm]{Example}
\newtheorem{examples}[thm]{Examples}
\newtheorem{corollary}[thm]{Corollary}
\newtheorem{rem}[thm]{Remark}
\newtheorem{lemma}[thm]{Lemma}
\newtheorem{cor}[thm]{Corollary}
\newtheorem{proposition}[thm]{Proposition}
\newtheorem{exs}[thm]{Examples}
\newtheorem{ex}[thm]{Example}
\newtheorem{exercise}[thm]{Exercise}
\numberwithin{equation}{section}%
\setcounter{part}{0}
\newcommand{\rra}{\longleftarrow}
\newcommand{\lra}{\longrightarrow}
\newcommand{\dra}{\Rightarrow}
\newcommand{\dla}{\Leftarrow}
\newtheorem{Thm}{Main Theorem}


\newtheorem*{thm*}{Theorem}
\newtheorem{lem}[thm]{Lemma}
\newtheorem*{lem*}{Lemma}
\newtheorem*{prop*}{Proposition}
\newtheorem*{cor*}{Corollary}
\newtheorem{dfn}[thm]{Definition}
\newtheorem*{defn*}{Definition}
\newtheorem{notadefn}[thm]{Notation and Definition}
\newtheorem*{notadefn*}{Notation and Definition}
\newtheorem{nota}[thm]{Notation}
\newtheorem*{nota*}{Notation}
\newtheorem{note}[thm]{Remark}
\newtheorem*{note*}{Remark}
\newtheorem*{notes*}{Remarks}
\newtheorem{hypo}[thm]{Hypothesis}
\newtheorem*{ex*}{Example}
\newtheorem{prob}[thm]{Problems}
\newtheorem{conj}[thm]{Conjecture}
\newcommand{\Res}{{\rm Res}}
\newcommand{\KL}{\unlhd}
\newcommand{\tto}{\twoheadrightarrow}
\newcommand{\xra}{\xrightarrow{\sim}}
\title{The primitive spectrum for $\mathfrak{gl}(m|n)$.}

\author{Kevin Coulembier\ff{Postdoctoral Fellow of the Research Foundation - Flanders (FWO)} \\Department of Mathematical Analysis, Ghent University\\ Department of Mathematics, University of California-Berkeley \\ email: {\tt
coulembier@cage.ugent.be}\\ \\ Ian M. Musson\ff{Research partly supported by  NSA Grant H98230-12-1-0249 and Simons Foundation grant 318264.} \\Department of Mathematical Sciences\\
University of Wisconsin-Milwaukee\\ email: {\tt
musson@uwm.edu}
}
\maketitle

\begin{abstract}
We study inclusions between primitive ideals in the universal enveloping algebra of general linear superalgebras. For classical Lie superalgebras, any primitive ideal is the annihilator of a simple highest weight module.  It therefore suffices to study the quasi-order on highest weights determined by the relation of inclusion between primitive ideals. For the specific case of reductive Lie algebras, this quasi-order is essentially the left Kazhdan-Lusztig quasi-order. For Lie superalgebras, the classification is unknown in general, safe from some low dimensional specific cases. We derive an alternative definition of the left Kazhdan-Lusztig quasi-order which extends to classical Lie superalgebras. We denote this quasi-order by $\KL$ and show that a relation in $\KL$ implies an inclusion between primitive ideals.

For $\mathfrak{gl}(m|n)$ the new quasi-order $\KL$ is defined explicitly in terms of Brundan's Kazhdan-Lusztig theory. We prove that $\KL$ induces an actual partial order on the set of primitive ideals. We conjecture that this is the inclusion order. By the above paragraph one direction of this conjecture is true. We prove several consistency results concerning the conjecture and prove it for singly atypical and typical blocks of $\mathfrak{gl}(m|n)$ and in general for $\mathfrak{gl}(2|2)$. An important tool is a new translation principle for primitive ideals, based on the crystal structure for category $\cO$. Finally we focus on an interesting explicit example; the poset of primitive ideals contained in the augmentation ideal for $\mathfrak{gl}(m|1)$.
\end{abstract}

\section{Introduction.} \label{uv.1}

The primitive spectrum for complex semisimple Lie algebras is an interesting and important mathematical structure, which has been well understood since about 1980. Duflo \cite{Du} proved that each primitive ideal is given by the annihilator ideal of a simple highest weight module. The  actual classification of primitive ideals was then completed by Borho, Dixmier, Garfinkle, Jantzen, Joseph and Vogan, details and references can be found in e.g. \cite{J2, M}. Their efforts led to a complete description of the poset  of primitive ideals;
 the final result was conjectured by Joseph in \cite{Jo} and proved by Vogan in \cite{Vo}. This description involves several reductions.  Using central characters, the poset decomposes as a disjoint union of finite connected components described by Weyl groups. The next step involves a reduction to the case of integral orbits of the Weyl group based on parabolic induction.  Finally, using translation to the walls, it remains to consider regular integral orbits. In this case the inclusions are governed by a partial quasi-order on the Weyl group, known as the left Kazhdan-Lusztig quasi-order (KL order for short) of \cite{KL}. 

In \cite{Jo, Vo}, there are two equivalent descriptions of this inclusion order for a regular orbit. The first is more direct and uses explicitly the Weyl group structure on the set of weights in a regular orbit, as well as the composition series of Verma modules. The second expresses the inclusion order  in terms of the projective functors on a regular block in category $\cO$, so actually by passing to the right KL order. The proof that the latter formulation is the correct description of the inclusion order relies heavily on the theory behind the equivalence of categories between regular blocks in category $\cO$ and Harish-Chandra bimodules (see \cite{BG}). The first formulation seems impossible to extend to Lie superalgebras, by lack of a proper Weyl group. The second formulation does not predict the correct inclusions for superalgebras, as we demonstrate in Subsection~\ref{sec55}. This is natural, as this formulation classically holds only for regular orbits. { It extends to a correct description of the primitive ideals corresponding to a regular block in parabolic category $\cO$, but not to a singular block in category $\cO$.} Hence it should not extend to atypical central characters for Lie superalgebras, which correspond to both regular and singular orbits. Also an equivalence with Harish-Chandra bimodules, of the type used in \cite{Vo}, has not been established for atypical blocks of Lie superalgebras, which is again natural as this particular equivalence also fails for singular blocks for Lie algebras either, see \cite{BG}.

There are some more extra difficulties in going from Lie algebras to Lie superalgebras. For instance it is impossible to reduce to finitely many integral blocks in category $\cO$, since blocks with similar characteristics (singularity and atyicality) will still not be equivalent, see e.g. \cite{CoSe}.
Furthermore it is possible for an infinite number of different primitive ideals to have the same central character,  and even sometimes for the poset of primitive ideals to have connected components containing infinitely many ideals.

For basic classical Lie superalgebras, the analogue of Duflo's result was established by the second author in \cite{M1}. For superalgebras of type I, the actual classification of the primitive ideals was completed by Letzter in \cite{L5}. An exhaustive list of inclusions was so far only obtained for the particular cases of $\mathfrak{sl}(2|1)$, $\mathfrak{osp}(1|2n)$ and $\mathfrak{q}(2)$ in \cite{M3, M5, Maq2}. Further techniques were developed by the first author and Mazorchuk in \cite{CoMa}, leading to partial results which will be extensively applied in the current paper. In particular all inclusions between primitive ideals in the generic region (far away from the walls of the Weyl chamber) and for typical weights were classified. This solves the bulk of the problem, but leaves open precisely the region where the behaviour is most complicated and interesting. Also one direction of the conjecture mentioned in the abstract was implicitly proved in \cite{CoMa}.

In the current paper we mainly focus on the primitive spectrum for $\mathfrak{gl}(m|n)$. In this case integral highest weights are labeled by elements of $\Z^{m|n}$, and we write $J(\ga)$ for the annihilator of the simple module corresponding to $\ga\in \Z^{m|n},$ see Section~2.
We make two major contributions to the study of the poset of primitive ideals for $\fgl(m|n).$
The first  is a translation principle for primitive ideals based on the translation functors introduced by Brundan \cite{Br} and studied further by Kujawa  \cite{Ku}. Even though simple modules are generally not mapped to simple modules, we construct, in Section 3, a translation principle for primitive ideals which preserves inclusions between certain sets of primitive ideals.  For semisimple Lie algebras, a translation principle for primitive ideals was introduced by Borho and Jantzen in \cite{BJ} and the reader might detect echoes of their work in the translation principle in the current paper. However, for $\mathfrak{gl}(m|n)$, the combinatorics is governed  by a crystal (in the sense of Kashiwara) rather than the Weyl group.

The second contribution is an alternative formulation of the left KL quasi-order, which can be extended to classical Lie superalgebras. Instead of relying on Weyl group combinatorics or projective functors, we find an alternative definition of the KL order, which uses the $\Ext^1$-quiver of a block in category $\cO$ (determined by validity of the KL conjecture of \cite{KL}, see e.g. \cite{BB}) and certain dominance conditions. One advantage of this definition is that it is directly applicable to singular blocks for Lie algebras. This means it describes {\it all} inclusions between annihilator ideals of integral simple highest weight modules directly, without the need for translation to the walls. Of course the fact that this predicts the correct inclusion order still relies on the results in \cite{Jo, Vo} and hence on the more standard formulations. The important feature for us is that the definition naturally extends to classical Lie superalgebras. From \cite{CLW, BLW}, we know that for the case $\mathfrak{gl}(m|n)$ the $\Ext^1$-quiver, and thus the KL order, is determined by Brundan's KL theory in \cite{Br}. We also study an analogue of the right KL order in Section 5.5 but show that this seems unrelated to the inclusion order.

We conjecture that  our left KL quasi-order, denoted here by $\KL$, induces the inclusion order on the set of primitive ideals for $\mathfrak{gl}(m|n)$.  The evidence in favour of this conjecture is presented as Theorem \ref{summary}. We show that $ \beta\KL \alpha$ implies that $J(\beta)\subseteq J(\alpha),$ and that $ J(\beta)= J(\alpha)$ iff $\beta\KL \alpha\mbox{ and } \alpha \KL \beta.$ The latter means the the equivalence classes, determined by our quasi-order, are the sets of modules with identical annihilator ideal.
We show that the conjecture is compatible with the  translation principle for primitive ideals and holds for ideals in the same Weyl group orbit. It follows also that the conjecture is correct in the generic region and for typical modules. Furthermore we show the conjecture holds for $\mathfrak{gl}(2|2)$ and for all singly atypical blocks for $\mathfrak{gl}(m|n)$, in particular implying the conjecture in full for $\mathfrak{gl}(m|1)$. As a side result we also obtain an algorithmic description of all inclusions for singly atypical blocks of $\mathfrak{gl}(m|n)$, in terms of the known inclusions for $\mathfrak{gl}(m)\oplus\mathfrak{gl}(n)$.
Note that our conjecture can be viewed as a  natural analog of the conjecture of Joseph for Lie algebras referred to earlier.

We also study, in Section 4, the poset of primitive ideals as a topological space, with respect to the Jacobson-Zariski topology. This aspect of the primitive spectrum has interesting features which do not appear for Lie algebras. The poset as a topological space is for instance no longer the disjoint union of its irreducible components. Closely related, the irreducible components of the topological space are not identical to the connected components of the poset. Nevertheless we will able to obtain a classification of the irreducible components of the topological space.

Finally, for $\mathfrak{gl}(m|1)$, we focus on an interesting special case, the poset and topological space $X$ corresponding to the primitive ideals included in the augmentation ideal. The motivation for this is given at the beginning of Section \ref{secaug}. As main results we show that $X$ is a connected component of the poset and that its $m$ irreducible components (as a topological space) are all isomorphic to the poset of primitive ideals of $U(\fg_0)$ at a regular integral central character.

\section{Preliminaries.}
\label{secprel}

For a basic classical Lie superalgebra (and for a reductive Lie algebra) $\fk$, we denote a Borel subalgebra by $\fb$ and a Cartan subalgebra by $\fh$. Denote the nilradical of $\fb$ by $\fn$, so $\fb=\fh\oplus\fn$. For any $\lambda\in\fh^\ast$, we denote the Verma module by $M_\lambda(\fk)=U(\fk)\otimes_{U(\fb)}\C_\lambda$. The top of this module is the simple highest weight module $L_\lambda(\fk)$. We denote the set of roots by $\Delta$ and the subset of positive roots by $\Delta^+$. We define $\rho(\fk)=\frac{1}{2}(\sum_{\gc \in\Delta^+_0}\gc )-\frac{1}{2}(\sum_{\gamma\in\Delta_1^+}\gamma)$. Let $P_0$, $P^{+}_0$, $P^{++}_0$ denote the set of integral, integral dominant and integral regular dominant weights respectively.

The BGG category will be denoted by~$\cO$. For
$\nu\in\fh^\ast$ and $N\in\cO$ we have
\begin{equation}\label{Vermancohom}\Ext^i_{\cO}(M_\nu(\fk),N)\cong \Hom_{\fh}(\C_\nu,H^i(\fn,N)),\end{equation}
see e.g. Theorem 25 (i) and Corollary 14 in \cite{CoMa2}. The full Serre subcategory of modules in the BGG category $\cO$ with integral weight spaces is denoted by $\cO_{\Z}$.

We are interested in primitive ideals. By
\cite{Du, M1} any primitive ideal in $U(\fk)$  has the form $I_\lambda(\fk) := \ann_{U(\fk)} (L_\lambda(\fk))$. More generally we set ${\cal X} = \{\ann  M| M \in \cal{O}\}.$ The set of primitive ideals in $U(\fk)$ is denoted by $\Prim U(\fk)\subset {\cal X}$. If there is a strict inclusion between two primitive ideals $I_\mu(\fk)$ and $I_\lambda(\fk)$ such that there is no third primitive ideal $I_\kappa(\fk)$ for which there are strict inclusions $I_\mu(\fk)\subset I_\kappa(\fk)\subset I_\lambda(\fk)$ we say that $I_\lambda(\fk)$ covers $I_\mu(\fk)$ and write $I_\mu(\fk)\prec I_\lambda(\fk)$.

We will also regard the set $\Prim U(\fk)$ as a topological space for the Jacobson-Zariski topology. Thus the closed sets are chosen to be
$$V(Q):=\{I\in \Prim U(\fk)\,|\, Q\subseteq I\},$$
for any two-sided ideal $Q$ in $U(\fk)$.

We will mainly focus on the case where $\fk$ is a general linear algebra. In this case we use the notation $\fg = \fgl(m|n)$. Unless stated otherwise, we take the Borel subalgebra $\fb$ corresponding to the distinguished system of positive roots
$\Delta^+=\Delta_0^+\cup
 \Delta_1^+$, where
\begin{equation*}\label{roots}
 \Delta_0^+=\{\gep_{i} - \gep_{j} |\hbox{ $1\le i< j\le m$}\} \cup \{
\gd_{i} - \gd_{j} |\hbox{ $1\le i< j\le n$}\},
\end{equation*}
\begin{equation*}\label{oroots}
 \Delta_1^+=\{\gep_{i} -\gd_j |\hbox{
$1\le i\le m$,
   $1\le  j\le n$}\}.
\end{equation*}

For this case we use the notation $L_\lambda=L_{\lambda}(\fg)$, $M_\lambda=M_\lambda(\fg)$, $J_\lambda=I_\lambda(\fg)$, $I_\lambda=I_\lambda(\fg_0)$, $U=U(\fg)$, $\rho=\rho(\fg)$ and $\rho_0=\rho(\fg_0)$.

We will often restrict to modules with integral weight spaces. The corresponding set of primitive ideals forms a subposet of $\Prim U$, which is not connected to the rest, we denote it by $\Prim_{\Z}U$.

We choose the form $(\cdot,\cdot)$ on $\fh^\ast$ by setting $(\epsilon_i,\epsilon_l)=\delta_{ij}$, $(\delta_j,\delta_k)=-\delta_{jk}$ and $(\epsilon_i,\delta_j)=0$. We have
\begin{equation}\label{rho}\gr=\frac{1}{2}\sum_{i=1}^m (m-n -2i+1)\gep_{i} +\frac{1}{2} \sum_{j=1}^n (n+m-2j+1)\gd_j.\end{equation}
It is sometimes more convenient to use
\[\pd=\sum_{i=1}^m (m-i)\gep_{i} +\sum_{j=1}^n (1-j)\gd_j \]
since the coefficients of $\pd$ are integers. The difference $\rho-\pd$
is orthogonal to all roots. The difference $\rho-\rho_0$ is orthogonal to all even roots.

We say that $\gl\in\fh^\ast$ is {\it singular} if $(\gl+\gr,\gamma^\vee)=(\lambda+\pd,\gc ^\vee)=(\lambda+\rho_0,\gc ^\vee) = 0$ for some $\gc \in\Delta^+_0$, with $\gamma^\vee:=2\gamma/(\gamma,\gamma)$. If $\gl$ is not singular it is {\it regular}. If $(\lambda+\rho,\gamma^\vee)\ge 0$, resp. $(\lambda+\rho,\gamma^\vee)\le 0$, for all $\gamma\in\Delta_0^+$, we say that $\lambda$ is {\it dominant}, resp. {\it anti-dominant}. If $\lambda$ is regular as well we say that it is {\it strictly} (anti-)dominant.

The {\it degree of atypicality} of $\lambda$ is the number of different mutually orthogonal odd roots $\gamma$ for which $(\lambda+\rho,\gamma)=(\lambda+\pd,\gamma)=0$. We say that $\lambda$ is {\it typical}, resp. {\it atypical} if the degree of atypicality is zero, resp. strictly greater than zero.

The $\rho$-shifted action of the Weyl group on $\fh^\ast$ is the same as the $\pd$-shifted or $\rho_{0}$-shifted action for $\mathfrak{gl}(m)\oplus\mathfrak{gl}(n)$, so
$$w\cdot\lambda=w(\lambda+\rho)-\rho=w(\lambda+\pd)-\pd=w(\lambda+\rho_0)-\rho_0.$$

We repeat the results in Theorems 6.1 and 11.1 of \cite{CoMa}, applied to $\mathfrak{gl}(m|n)$ with system of positive roots as above.

\begin{theorem}
\label{thmCoMa}
Consider $\lambda,\mu\in \fh^\ast$, then  we have
\begin{enumerate}[$(i)$]
\item $J_\mu=J_\lambda\quad \Leftrightarrow  \quad I_\mu=I_\lambda$;
\item $J_{w'\cdot\lambda}\subset J_{w\cdot\lambda}\quad \Leftrightarrow  \quad I_{w'\cdot\lambda}\subset I_{w\cdot\lambda}$ for $w,w'\in W$;
\item If $\kappa\in\fh^\ast$ is typical, then $J_\lambda\subset J_\kappa$ or $J_\kappa\subset J_\mu$ imply $I_\lambda\subset I_\kappa$ and $I_\kappa\subset I_\mu$.
\end{enumerate}

\end{theorem}
\noi Property (i) was first proved by Letzter in \cite{L5}. Property (iii) is actually a special case of property (ii), based on central character arguments.

We fix a bijection between integral weights $P_0\subset\fh^\ast$ and $\Z^{m|n}$, by
\be \label{hat}P_0\,\,\tilde{\to}\,\, \Z^{m|n},\quad\lambda\mapsto \alpha^\lambda\quad \mbox{with}\quad \alpha^\lambda_i=(\lambda+\pd,\epsilon_i)\quad\mbox{and}\quad \alpha^\lambda_{m+j}=(\lambda+\pd,\delta_j).\ee
Elements of $\Z^{m|n}$ are denoted by $(\alpha_1,\cdots,\alpha_m|\alpha_{m+1},\cdots,\alpha_{m+n})$, where $|$ is referred to as the separator.

We use the notation $L(\alpha^\lambda):=L_\lambda$ and $J(\alpha^\lambda):=J_\lambda$ for any $\lambda\in P_0$. The dot action of the Weyl group $W$ on $P_0$ corresponds to the regular action of $W\cong S_{m}\times S_n$ on $\Z^{m|n}$. The longest element of $W$ is denoted by $w_0$.

We will need some results on the primitive spectrum of a reductive Lie algebra~$\fk$, see \cite{M} Section 15.3 or \cite{J2}.
For $\lambda \in \fh^*$, let ${\mathscr{X}}_{\lambda}$ denote the subset of Prim $U(\fk)$ consisting of primitive ideals containing the kernel of the
central character determined by $\gl.$ Let $B$ be the set of simple roots of $\fk.$  For $w\in W$, and $\mu \in \fh^*$ set
$\gt(w) =\{\ga\in B|w\ga<0\},$ and ${B}^0_{\mu} =\{\ga \in B|(\mu+\gr(\fk),\ga)= 0\}.$

\noi For any (possibly singular) $\lambda\in P_0$ we will write $\tau(\lambda)$ for $\tau(w)$ with $w\in W$ the longest element of the Weyl group for which $w^{-1}\cdot \lambda$ is dominant.  Thus for $\kappa\in P_0^{++}$, we just have $\tau(w\cdot\kappa)=\tau(w)$ for any $w\in W$.
\bt \label{tppi}
Consider $\fk$ a reductive Lie algebra. \bi
\itemi Any primitive ideal in  $U(\fk)$ has the form $I_\gl(\fk)$ for some $\gl \in \fh^*.$
\itemii If $\lambda \in P^{++}_0$, there is a well defined map from ${\mathscr{X}}_{\lambda}$ to the  power set of $B$, sending $I_{w\cdot \gl}(\fk)$ to $\gt(w)$. This map is surjective and order-reversing.
\itemiii If $\lambda \in P^{++}_0$ and $\mu \in P_0^+$,
then there is an isomorphism of posets
\[ \psi : \{ I \in {\mathscr{X}}_{\lambda}| {B}^0_{\mu} \subseteq \tau(I) \}
\stackrel{{\longrightarrow}} \; {\mathscr{X}}_{\mu}. \]
If $w \in W_{\lambda}$ and ${B}^0_{\mu} \subseteq \tau(w)$, then $\psi (I_{w\cdot\lambda}(\fk)) = I_{w\cdot\mu}(\fk)$.
\ei
\et

\noi As in definition 11.5 of \cite{CoMa}, for any $\alpha\in\Z^{m|n}$ we set
\be \label{ddef}
d_\alpha = \max\{ k \in \bbZ_+ | \mbox{ there are }  \gc_1, \ldots, \gc_k \in \Gd_1^+
\mbox{ with } e_{-\gc_1} \ldots  e_{-\gc_k}v_\alpha \neq 0\},\ee
where $v_\alpha$ represents the highest weight vector of the module $L(\alpha)$. We fix an element $h$  in the center of $\fg_0$ such that the adjoint action on $\fg_{1}$ is given by $+1$ and on $\fg_{-1}$ by $-1$. By definition we therefore have that the number of different eigenvalues of $h$ on $L(\alpha)$ is equal to $d_\alpha+1.$ Note that we have $d_\alpha\le mn$, where the equality is reached if and only if $\alpha$ is typical.

We will use the concept of odd reflections, see e.g. \cite{M, Vera}. We will only use this for the case $\mathfrak{gl}(m|1)$, so we use the corresponding notation here. In particular we are interested in going from the distinguished system of positive roots to the antidistinguished system, i.e. the one with positive roots $\Delta_0^+\cup (-\Delta_1^+)$. There is a  sequence
 \be \label{distm} \mathfrak{b}^{(0)}, \mathfrak{b}^{(1)}, \ldots,
 \mathfrak{b}^{(m)}. \ee
  of Borel subalgebras such that $\mathfrak{b}^{(0)}$ is distinguished, $\mathfrak{b}^{(m)}$ is antidistinguished and
	$\mathfrak{b}^{(i-1)},\mathfrak{b}^{(i)}$ are adjacent for $1 \leq i \leq m$. There are isotropic roots $\ga_i=\epsilon_{m-i+1}-\delta$ such that $\fg^{\ga_i} \subset \fb^{(i-1)}, \fg^{-\ga_i} \subset \fb^{(i)}$  for $1 \leq i \leq m$, and  $\ga_1,\ldots,\ga_m$ are the distinct odd positive roots of~$\mathfrak{gl}(m|1).$  For any $\lambda\in\fh^\ast,$ we define $\lambda^{\ad}\in\fh^\ast$ as the highest weight of the simple module $L_\lambda$ with respect to the antidistinguished system of positive roots, so $L_{\gl} = L^{\ad}_{\gl^{\ad}}$.

Finally recall that when $\fg = \fsl(m)$ (or $\fgl(m)$) the set of primitive ideals with a regular integral central character can be described using the Robinson-Schensted correspondence. To fix notation we will use the bijection
\be \label{elf}  v \longrightarrow (A(v), B(v)) \ee
from the symmetric group ${S}_m$ to the set of all
pairs of standard tableaux with $m$ boxes, having  the same shape as defined in \cite{M} Theorem 11.7.1, see also \cite{J2} Section 5.24.  Then we have by \cite{M} Theorem 15.3.5 that for $\mu \in P_0^{++}$,   $I_{u \cdot \mu} = I_{v \cdot \mu}$ if and only if $A(u) = A(v).$ Note that $v$ is an involution, that is $v^2=1$, iff $A(v)= B(v)$ in \eqref{elf}. Hence any ideal contained in $I_{\mu}$ has the form $I_{v \cdot \mu}$ for a unique involution $v$.


\section{A translation principle for primitive ideals.}

In this section we introduce a translation principle on the poset of primitive ideals for $\mathfrak{gl}(m|n)$. In Subsection \ref{sscrystals} we review the crystal structure introduced by Brundan. In Subsection \ref{sstf} we derive some immediate consequences of the results on translation functors by Kujawa. This is then used in Subsection \ref{Apollo} to introduce the translation principle.

\subsection{Crystals.}\label{sscrystals}
First, we define a crystal
$(\mathbb{Z}^{m|n}, \tilde e_i, \tilde f_i, \varepsilon_i, \phi_i)$
in the sense of Kashiwara \cite{Ka}, as introduced by Brundan in \cite{Br}.
Take $i \in \mathbb{Z}$ and
$$\alpha = (a_1,\dots,a_m|a_{m+1},\cdots,a_{m+n}) \in \mathbb{Z}^{m|n}.$$
The {\em $i$-signature} of $\alpha$ is the
tuple $(\sigma_1,\dots,\sigma_m|\sigma_{m+1},\cdots,\sigma_{m+n})$ defined by:
\begin{eqnarray*}
&&\bullet\, \mbox{for }\, j\le m: \qquad\sigma_j = \left\{
\begin{array}{ll}
+&\hbox{if $a_j = i$,}\\
-&\hbox{if $a_j = i+1$,}\\
0&\hbox{otherwise;}
\end{array}\right.\\
&&\bullet\, \mbox{for }\, j> m: \qquad\sigma_j = \left\{
\begin{array}{ll}
+&\hbox{if $a_j = i+1$,}\\
-&\hbox{if $a_j = i$,}\\
0&\hbox{otherwise.}
\end{array}\right.
\end{eqnarray*}

We use the  crystal operators on $\Z^{m|n}$ defined in \cite{Br} beginning with equation (2.32). The {\em reduced $i$-signature} of $\ga$ is obtained from $i$-signature of $\alpha$ by successively replacing
sequences of the form
$-+$ (possibly separated by $0$'s) with $00$
until no $-$ appears to the left of a $+$.
\\ \\
We introduce $c_j$ to denote $(0,\dots,0,\pm 1,0,\dots,0) \in \Z^{m|n}$
where $\pm 1$ appears in the $j$th place as $1$ if $j\le m$ and as $-1$ if $j>m$.
Define
\begin{align}\nonumber
\tilde e_i(\alpha) &:= \left\{
\begin{array}{ll}
\emptyset&\hbox{if there are no $-$'s in the reduced $i$-signature},\\
\alpha - c_j&\hbox{if the leftmost $-$ is in position $j$;}
\end{array}\right.\\\nonumber
\tilde f_i(\alpha) &:= \left\{
\begin{array}{ll}
\emptyset&\hbox{if there are no $+$'s in the reduced $i$-signature},\\
\alpha + c_j&\hbox{if the rightmost $+$ is in position $j$;}
\end{array}\right.\\\nonumber
\varepsilon_i(\alpha) &= \hbox{the total number of $-$'s in the
reduced $i$-signature};\\\nonumber
\phi_i(\alpha) &= \hbox{the total number of $+$'s in the reduced
$i$-signature}.
\end{align}
Consequently, the reduced signature of $\tilde{e}_i (\alpha)$ is obtained from the reduced signature of $\alpha$ by replacing the leftmost $-$ by $+$. This implies that for $\alpha \in \mathbb{Z}^{m|n}$, we have that
$$ \varepsilon_i(\ga) = \max\{r \geq 0\:|\:(\tilde e_i)^r(\ga) \neq \emptyset\},$$
$$\phi_i(\ga) = \max\{r \geq 0\:|\:(\tilde f_i)^r(\ga) \neq \emptyset\}.$$
Note that by definition we have $\sum_{i\in\Z}\varepsilon_i(\alpha)=\sum_{i\in\Z}\phi_i(\alpha)$.

\subsection{Translation functors.}
\label{sstf}

In this subsection we demonstrate how the action of translation functors on the integral BGG category $\cO_{\Z}$ can be linked to the crystals in the previous subsection. This is an immediate consequence of Kujawa's result in Theorem 2.4 of \cite{Ku} together with general results in \cite{Br,ChR}.

Denote the tautological representation of $\mathfrak{gl}(m|n)$ by $E=\C^{m|n}$. For an arbitrary central character $\chi$ we set $\chi'=\chi_{\tilde e_i\alpha}$ and $\chi''=\chi_{\tilde f_i\alpha}$ for any $\alpha$ such that $\chi_{\alpha}=\chi$ and $\tilde e_i\alpha\not=\emptyset$ or $\tilde f_i\alpha\not=\emptyset$. For $M\in\cO_\chi$ we set
\begin{eqnarray*}\quad e_i(M)=(M\otimes E^\ast)_{\chi'}\quad \mbox{and}\quad f_i(M)=(M\otimes E)_{\chi''}.
\end{eqnarray*}

\begin{theorem}
\label{thmKujCR}
Let $\alpha \in \Z^{m|n}$ and $i \in \Z$.
\begin{enumerate}
\item[\rm(i)]
Set $s=\varepsilon_i(\alpha)$, if $s = 0$ then $e_i L(\alpha) = 0$. If $s=1$, then $e_i L(\alpha)\cong L(\tilde e_i\alpha)$. If $s>1$, then $e_i L(\alpha)$ is not simple but an indecomposable module with
irreducible socle and top isomorphic to $L(\tilde e_i(\alpha))$.

Furthermore, any simple subquotient of $e_i L(\alpha) = 0$, different from $L(\tilde e_i(\alpha))$, is of the form $L(\beta)$ with $e_i^{s-1}L(\beta)=0$.
\item[\rm(ii)]
Set $t=\phi_i(\alpha)$, if $t = 0$ then $f_i L(\alpha) = 0$. If $t=1$, then $f_i L(\alpha)\cong L(\tilde f_i\alpha)$. If $t>1$, then $f_i L(\alpha)$ is not simple but an indecomposable module with
irreducible socle and top isomorphic to $L(\tilde f_i(\alpha))$.

Furthermore, any simple subquotient of $f_i L(\alpha) = 0$ is of the form $L(\beta)$ with $f_i^{t-1}L(\beta)=0$.
\end{enumerate}
\end{theorem}
\begin{proof}
We only prove (i), since  (ii)  is proved in the same way. The first paragraph is precisely Theorem 2.5(i) in \cite{Ku}.
We consider the Lie algebra $\mathfrak{sl}(\infty)$ with tautological representation $V$. The canonical basis of $V$ is labeled by $\Z$. This extends to a mapping from the vector space $\Z^{m|n}$ to the $\mathfrak{sl}(\infty)$-representation $(\otimes^m V)\otimes (\otimes^n V^\ast)$. The identification $M(\alpha)\leftrightarrow \alpha$ for $\alpha\in\Z^{m|n}$ yields a bijection between the Grothendieck group $K(\cO_{\Z}^{\Delta})$ of the category of modules in $\cO_{\Z}$ with Verma flag and $(\otimes^m V)\otimes (\otimes^n V^\ast)$. Under this bijection atypical simple modules are not in $(\otimes^m V)\otimes (\otimes^n V^\ast)$, but in a completion. In order to fix this, we need to restrict to some finite interval $I\subset\Z$. We use the notation of \cite{BLW}. The algebra $\mathfrak{sl}(I)$ is generated by $\{e_i,f_i |\in I\}$ and this yields a categorification of a corresponding subquotient $\cO_I$. Now there is a bijection $$(\otimes^m V_I)\otimes (\otimes^n V_I^\ast)\leftrightarrow K(\cO_I).$$
Theorem 4.28 in \cite{Br} then implies that the translation functors $\tilde e_i$ and $\tilde f_i$ for a fixed $i\in I$ act on $(\otimes^m V)\otimes (\otimes^n V^\ast)$ (or the corresponding tensor space for $\mathfrak{sl}(I)$) as the Chevalley generators of $\mathfrak{sl}(2)$, yielding a categorification.

The results in \cite{BLW} imply that this construction is well-behaved with respect to the limit $I\to \Z$. The second paragraph is therefore an immediate consequence of Lemma 4.3 in \cite{ChR}, see also Theorem 4.4 in \cite{BruK}.
\end{proof}

Note that this means that if $\tilde e_i\alpha\not=0$ we have the property
\begin{equation}
\label{Angliru}
\tau(\tilde e_i\alpha)=\tau (\alpha).
\end{equation}

As simple modules in category $\cO$ have no  self-extensions we obtain, by induction, the following consequence of Lemma \ref{thmKujCR}.
\begin{corollary}
\label{corKujCR}
Let $\alpha \in \Z^{m|n}$ and $i \in \Z$.
\begin{enumerate}
\item[\rm(i)]
If $s=\varepsilon_i(\alpha)>0$,
then $e_i^{s+1} L(\alpha)=0$ while $e_i^{s} L(\alpha)$ is isomorphic to a non-zero direct sum of simple modules isomorphic to $ L(\tilde e_i^{s}(\alpha))$.
\item[\rm(ii)]
If $s=\phi_i(\alpha)>0$, then $f_i^{s+1} L(\alpha)=0$ while $f_i^{s} L(\alpha)$ is isomorphic to a non-zero direct sum of simple modules isomorphic to  $L(\tilde{f}_i^{s}(\alpha))$.
\end{enumerate}
\end{corollary}


The following remark is immediate but will be useful for later purposes.
\br\label{Dionysos}
Consider $\alpha,\beta\in\Z^{m|n}$ with $\chi_\alpha=\chi_\beta$. If $\tilde e_i \alpha\not=\emptyset$ and $\tilde e_i\beta\not=\emptyset$ $($respectively $\tilde f_i \alpha\not=\emptyset$ and $\tilde f_i\beta\not=\emptyset)$ we have $\chi_{\tilde e_i \alpha}=\chi_{\tilde e_i\beta}$ $($respectively $\chi_{\tilde f_i \alpha}=\chi_{\tilde f_i\beta}  )$.
\er


\subsection{Translation functors and primitive Ideals.}
\label{Apollo}
Now we can discuss the translation principle for primitive ideals which are annihilators of highest weight modules in the integral block. This restriction to integral weights is partly justified by the classical case, the results in \cite{CMW} and Corollary~8.4 in \cite{CoMa}.

It is easy to see that for all $i \in \Z$, there are well defined map of posets
$E_i':{\cal X} \longrightarrow {\cal X}$ given by
$E_i'(\ann  M) = \ann  e_i (M)$, see \cite{J2} Lemma 5.4, or Lemmata 4.1 and 4.3 in \cite{CoMa}. According to Corollary \ref{corKujCR} we have

\[\varepsilon_i(\ga) = \max\{ n| e_i^n  L({\ga}) \neq 0 \} \] and hence
\be \label{idist} \varepsilon_i(\ga) = \max\{ n| \ann (e_i^n  L(\ga)) \neq U \}= \max\{ n| (E_i')^n  J(\ga) \neq U \} \ee and this depends only on the ideal $J(\ga).$

\bl \label{bat} If $J(\gb) \subseteq J(\ga),$ then $\varepsilon_i(\gb) \geq \varepsilon_i(\ga)$ and $\phi_i(\beta)\ge \phi_i(\alpha)$, for each $i\in\Z$. \el
\bpf We use equation (\ref{idist}).  If $k = \varepsilon_i(\gb),$ then
$$U = \ann e^{k+1}_i L({\gb})=(E_i')^{k+1} J(\beta)\subseteq (E_i')^{k+1} J(\alpha)= \ann e^{k+1}_i L({\ga}),$$ and it follows that $\varepsilon_i(\ga) \leq k.$ The result for $\phi_i$ is proved similarly.
\epf

\begin{corollary}
\label{vuelta}
If for $\alpha,\beta,\kappa\in\Z^{m|n}$ we have $J(\beta)\subset J(\kappa)\subset J(\alpha)$ and $\varepsilon_i(\alpha)=\varepsilon_i(\beta)$ and $\phi_i(\alpha)=\phi_i(\beta)$ for some $i$, then $\varepsilon_i(\kappa)=\varepsilon_i(\alpha)$ and $\phi_i(\kappa)=\phi_i(\alpha)$.
\end{corollary}
\noi
In general, the map $E'_i$ does not take primitive ideals to primitive ideals.   Instead we define $E_i:\Prim U\to\Prim U$ by setting $$E_i J(\alpha)=E_i(\ann  L(\ga)) := \ann  soc(e_i (L({\ga}))) =J(\tilde e_i \ga),$$
where we used Theorem \ref{thmKujCR}(i). In the same way we define $F_i$ from $f_i$. Set
$$\Prim^{(i)}_{r,s} U = \{J(\ga)| \varepsilon_i(\ga) =r, \phi_i(\ga) =s \}\subset \Prim U\subset {\cal X}.$$
\bt \label{thm1}
If $r \geq 1$ and $s\ge 0$, the map ${E}_i$ gives a well defined isomorphism of posets
$$\Prim_{r,s}^{(i)} U \longrightarrow  \Prim^{(i)}_{r-1,s+1} U,$$
with inverse $F_i$.
Moreover, for $r,s \geq 1,$ the maps
$$E_i:\cup_{t\ge 0}\Prim_{r,t}^{(i)} U \to \cup_{t\ge 1} \Prim^{(i)}_{r-1,t} U,\quad F_i:\cup_{t\ge 0}\Prim_{t,s}^{(i)} U \to \cup_{t\ge 1} \Prim^{(i)}_{t,s-1} U,$$
are bijective and preserve inclusions.
\et
\bpf
The fact that $E_i$ maps bijectively from $\Prim_{r,s} U$ to $\Prim_{r-1,s+1} U$ follows from Corollary \ref{corKujCR}. Now we prove that $E_i:\cup_{t\ge 0}\Prim_{r,t}^{(i)} U \to \cup_{t\ge 1} \Prim^{(i)}_{r-1,t} U$ preserves inclusions.

Suppose that $\ga, \ga'$ are such that $J(\ga) \subseteq  J(\ga')$ and $\varepsilon_i(\alpha)=\varepsilon_i(\alpha')=r$.  Set $\tilde e_i(\ga) = \gb$ and   $\tilde e_i(\ga') = \gb'.$ We write $\rad(J)$ for the radical of an ideal $J.$ For any $\fg$-module $M$ of finite length, $\rad(\ann  M)$ is the intersection of the annihilators of the composition factors of $M.$ We thus have $$\rad(\ann ( e_i L({\ga}))) \subseteq \rad(\ann ( e_i L({\ga'}))) \subseteq J({\gb'}),$$
where the second inequality follows from Theorem \ref{thmKujCR}(i). In addition, Theorem~\ref{thmKujCR}(i) implies that there is a set $S\subset\Z^{m|n}$, where $\gamma\in S$ implies $\varepsilon_i(\gamma)<\varepsilon_i(\beta)$, such that
\be \label{catman} \rad(\ann ( e_i L({\ga}))) = J({\gb}) \cap \bigcap_{\gamma\in S} J(\gc).\ee
The product of the ideals on the right side of (\ref{catman}) is thus contained in $ J({\gb'}).$  Since $J({\gb'})$ is prime, one of these ideals is contained in $J({\gb'})$.
If $J({\gc}) \subseteq J({\gb'})$ for some  $\gc\in S,$ then Lemma \ref{bat} implies $\varepsilon_i(\gamma)\ge \epsilon_i(\beta')=r-1=\varepsilon_i(\beta)$, a contradiction.
Therefore $J({\gb}) \subseteq J({\gb'}).$ The same reasoning for $F_i$ concludes the proof.
\epf

\section{The irreducible components of the topological space}

In this subsection we obtain, as immediate application of our results in Subsection~\ref{Apollo}, a classification of the irreducible components of the space $\Prim_{\Z} U$ with respect to the Jacobson-Zariski topology. First we state some immediate facts about the corresponding topological space for Lie algebras.
\begin{proposition}
Consider $\fk$ a reductive Lie algebra. Then the irreducible components of the topological space $\Prim _{\Z}U(\fk)$ are the same as the connected components of $\Prim_{\Z} U(\fk)$ as a poset. These are in one to one correspondence with integral central characters, or dominant weights $\lambda$ and given by
$$\{I\in\Prim_{\Z} U(\fk)\,|\, I_{w_0\cdot \lambda}\subset I \}=\{I\in\Prim_{\Z} U(\fk)\,|\, I\subset I_{\lambda}\}.$$
\end{proposition}
Almost all these properties no longer hold for $\mathfrak{gl}(m|n)$, see Proposition \ref{topnew}, but the connection between irreducible components and (anti-)dominant weights is still valid as stated in the following theorem.

\begin{theorem}
\label{thmcomp}
The irreducible components of the topological space $\Prim_{\Z} U$ are
$$Z(\beta)=\{J\in \Prim U\,|\, J(\beta)\subset J\},$$
for all anti-dominant $\beta\in\Z^{m|n}$.
\end{theorem}
\begin{proof}
The fact that $Z(\beta)$ for $\beta$ anti-dominant is irreducible is immediate. It remains to be proven that $Z(\beta)$ is maximal. Therefore we claim that there are no primitive ideals properly included in $J(\beta)$ if $\beta$ is anti-dominant, from which this statement follows. If there would be a proper inclusion $J(\gamma)\subset J(\beta)$, without loss of generality we can assume that $\gamma$ is anti-dominant by Theorem \ref{thmCoMa} (ii). The claim therefore follows from the subsequent Lemma \ref{joueur}.
\end{proof}

\begin{lemma}
\label{joueur}
Consider antidominant $\beta,\gamma\in\Z^{m|n}$, then an inclusion $J(\gamma)\subseteq J(\beta)$ implies $\gamma=\beta$.
\end{lemma}
\begin{proof}
We introduce some notation, for any $\alpha\in\Z^{m|n}$ and $x\in\Z$ we set $\alpha_0(x)$ equal to the number of labels left of the separator equal to $x$ and $\alpha_1(x)$ equal to the number of labels right of the separator equal to $x$. If $\alpha$ is anti-dominant we have
\begin{eqnarray*}
\phi_x(\alpha)&=&\alpha_0(x)+\max(\alpha_1(x+1)-\alpha_0(x+1),0)\\
\varepsilon_x(\alpha)&=&\alpha_1(x)+\max(\alpha_0(x+1)-\alpha_1(x+1),0).\end{eqnarray*}
For arbitrary $\beta,\gamma\in\Z^{m|n}$ that satisfy $\chi_\beta=\chi_\gamma$ we have $\beta_0(y)-\beta_1(y)=\gamma_0(y)-\gamma_1(y)$ for any $y\in\Z$.

Applying the considerations in the previous paragraph to two anti-dominant $\beta,\gamma$ with the same central character yields  $$\phi_x(\gamma)-\phi_x(\beta)=\gamma_0(x)-\beta_0(x)\quad\mbox{ and }\quad \varepsilon_x(\gamma)-\varepsilon_x(\beta)=\gamma_1(x)-\beta_1(x).$$ Lemma \ref{bat} implies that an inclusion $J(\gamma)\subseteq J(\beta)$ would thus imply $\gamma_0(x)\ge\beta_0(x)$ and $\gamma_1(x)\ge\beta_1(x)$ for all $x\in \Z$. As we have
$$\sum_x \gamma_0(x)=\sum_x\beta_0(x)=m\quad\mbox{and}\quad \sum_x \gamma_1(x)=\sum_x\beta_1(x)=n, $$
we come to the conclusion that $ \gamma_0(x)=\beta_0(x)$ and  $\gamma_1(x)=\beta_1(x)$. As both $\gamma$ and $\beta$ are anti-dominant we find $\beta=\gamma$.
\end{proof}

\begin{proposition}
\label{topnew}
In general, the irreducible components are non-trivial subsets of the connected components of $\Prim_{\Z}U$. The irreducible components can possess more than one maximal element as a poset. The connected components can possess more than one maximal and more than one minimal element.
\end{proposition}
\begin{proof}
The connected component of $\Prim_{\Z}U(\mathfrak{gl}(2|2))$ containing the augmentation ideal, considered in Subsection \ref{sec22}, provides an example for all of these features.
\end{proof}


\section{A super analogue of the left Kazhdan-Lusztig order.}

In this section we study analogues of the left and right Kazhdan-Lusztig quasiorder on the Weyl group in the context of Lie superalgebras. We find that our analogue of the left order seems a good candidate to describe the inclusion order, supported by an extensive list of correspondences in Theorem \ref{summary}, whereas the right order has a very different nature. In particular the right order is not interval finite, whereas the inclusion order is interval finite, as is the left order.
\subsection{An alternative description of the primitive spectrum of a semisimple Lie algebra.}
\label{secKLorder}
We fix a reductive Lie algebra $\fk$. Recall that a {\it quasi-order} on a set is a relation that is reflexive and transitive. We denote the partial ordering on $P_0$ corresponding to the dominance order by $\le$. We define~$\KL$ as the smallest quasi-ordering on $P_0$ such that for $\lambda,\nu\in P_0$ and a simple reflection $s\in W$, we have $\nu \KL \lambda$ if
\begin{itemize}
\item[\rm (i)] $s\cdot \lambda <\lambda$ and $s\cdot\nu\ge \nu$;
\item[\rm (ii)] $\Ext^1_{\cO}(L_\lambda(\fk),L_\nu(\fk))\not=0.$
\end{itemize}
 Using Kazhdan-Lusztig theory we reformulate property (ii) in terms of extensions with Verma modules, see equation \eqref{extVerma} below. In particular the value
$$\mu(\lambda,\nu):=\dim\Ext^1_{\cO}(L_\lambda(\fk),L_{\nu}(\fk))$$
is known as the Kazhdan-Lusztig $\mu$-function, see \cite{KL} and Section 2.1 in \cite{Ma09}. The $\mu$-function can in turn be expressed through equation \eqref{Vermancohom} in terms of cohomology of the nilradical of the Borel subalgebra.
\begin{theorem}
\label{reformVogan}
For any $\lambda,\mu\in P_0$, we have
$$I_\mu(\fk)\subseteq I_\lambda(\fk)\Leftrightarrow \mu \KL \lambda.$$
Consequently, for $\kappa\in P_0^{++}$ and $w,w'\in W$, we have
$$w\cdot \kappa \KL w'\cdot \kappa \quad \Leftrightarrow \quad w' \preceq^{(l)}_{KL} w,$$
with $\preceq^{(l)}_{KL}$ the left Kazhdan-Lusztig order, see \cite{J2, Jo, MaMi}.
\end{theorem}
\noi The proof is based  on the next lemma, which is well-known to specialists. We include a proof
for completeness. It is possible to give a proof of Theorem \ref{reformVogan} avoiding the use of this lemma, but that requires the fact that twisting functors and coshuffling functors are Koszul dual.
\bl \label{gnat} For $x, y \in W$ we have
\begin{equation}\label{Artemis}\dim\Ext_{\cO}^1(L_{x\cdot 0}(\fk),L_{y\cdot 0}(\fk))=\dim\Ext^1_{\cO}(L_{x^{-1}\cdot0}(\fk),L_{y^{-1}\cdot 0}(\fk)).\end{equation}
\el
\bpf
In the proof we leave out the references to $\fk$ in notation such as $L_\lambda(\fk)$ and $I_\lambda(\fk)$. Let $\cH$ denote the category of Harish-Chandra bimodules that admit generalised trivial central character on both sides. Also let $\cH^1$, ${}^1\cH$ and $\bar \cH ={}^1\cH^1$ stand for the full subcategories of $\cH$ of the modules that admit trivial central character on respectively the right side, left side and both sides.
\\ \\
 For $M, N$ objects of $\cO_0$, let $\cL(M,N)$  denote the submodule of $\Hom(M,N)$ consisting of maps which are locally finite under the diagonal action of $\fg$. By \cite{BG} or \cite{J2} 6.27 the functor $N \lra \cL(M(0), N)$ provides an equivalence of categories from $\cO_0$ to $\cH^1$. By restriction, we obtain an equivalence between the full subcategory of $\cO_0$  consisting of modules that admit the trivial central character, and the category $\bar \cH $.
The extensions in \eqref{Artemis} correspond to modules which are quotients of Verma modules (or submodules of dual Verma modules) and therefore admit a central character. Under the equivalence between $\cO_0$ and $\cH^1$, these modules are therefore inside the full subcategory $\bar \cH$.
For $x\in W$ set $\ttl_x = \cL(M(0),L_{x\cdot 0})$. It follows that \eqref{Artemis} is equivalent to the following
\begin{equation}\label{Art}\dim\Ext_{\bar \cH }^1(\ttl _{x},\ttl _{y})=\dim\Ext^1_{\bar \cH }(\ttl _{x^{-1}},\ttl _{y^{-1}}).\end{equation}
Let ${u \lra} {^tu}$  denote the antiautomorphism of $\fg$ defined in \cite{J2} 2.1, or \cite{M} Proposition 8.6.1.
As in \cite{J2} 6.3, given a $U(\fg)$ bimodule $X$, we can define a new bimodule $^s X$ which is equal to $X$ as a vector space, with a new action $*$ given by
\[{u_1 * m * u_2 =} {^tu_2}m {^tu_1} \quad \mbox{ for all } u_1, u_2 \in U(\fg),\; m \in X. \]
 The map $\eta:{X \lra} \; {^sX}$  is an equivalence from  $\cH^1$ to ${}^1\cH$, preserving $\bar \cH $ and yielding $\eta(\ttl_{x})\cong \ttl_{x^{-1}}$, see Satz 6.34 in \cite{J2}. So \eqref{Art} follows from this.
\epf
\noi  {\it Proof of Theorem \ref{reformVogan}.}
We prove this statement first for regular blocks (i.e. the principal block $\cO_0$). The poset for this highest weight category is $\{w\cdot0\,|\,w\in W\}$. We use the convention $y< x$ iff $x\cdot0 < y\cdot 0$ and $y\KL x$ iff $x\cdot0 \KL y\cdot 0$.
\\ \\
By \eqref{Artemis} we can reformulate the generating condition for $\KL$ on the Weyl group by taking $\lambda=x\cdot 0$ and $\nu=y\cdot 0$ as follows. The quasi-order $\KL$ on $W$ is defined as the smallest quasi-order such that $y\KL x$ if
\begin{itemize}
\item[\rm (a)] $x^{-1} < x^{-1}s$ and $y^{-1}s< y^{-1}$;
\item[\rm (b)] $\Ext^1_{\cO}(L_{x^{-1}\cdot 0}(\fk),L_{y^{-1}\cdot 0}(\fk))\not=0.$
\end{itemize}
According to equation (2.2) in \cite{Ma09} these two conditions equal
$$[\theta_s L_{x^{-1}\cdot 0}:L_{y^{-1}\cdot 0}]\not=0,$$
with $\theta_s$ the translation through the $s$-wall. Lemma 13 in \cite{MaMi} and Corollary 7.13 in \cite{J2} therefore imply $y\KL x\Leftrightarrow I_{x\cdot 0}\subseteq I_{y\cdot 0}.$ So we find $\KL$ is equal to $\preceq^{(l)}_{KL}$ on $W$.

It remains to prove the statement for singular blocks. The property $\mu\KL\lambda\Rightarrow J_\mu\subseteq J_\lambda$ follows from Lemma 5.17 in \cite{CoMa} applied to Lie algebras. We prove the other direction. Consider an (integral) singular block, with $T$ the translation functor from a regular block to our singular block, $\widetilde{T}$ its adjoint and $\theta=\widetilde{T}T$ the translation through the wall. Each highest weight $\lambda$ for the singular block has a unique highest weight $\lambda'$ for the regular block such that $TL(\lambda')=L(\lambda)$. According to Theorem \ref{tppi} we have
$$I_{\mu}\subseteq I_\lambda\quad\Leftrightarrow\quad I_{\mu'}\subseteq I_{\lambda'}.$$
The proof is therefore completed if we prove that conditions (i) and (ii) hold for $\mu,\lambda$ if they hold for $\mu',\lambda'$. This is trivial for condition (i). For (ii), assume that $\Ext^1_{\cO}(L_{\mu'},L_{\lambda'})\not=0$. We have
$$\Ext^1_{\cO}(L_{\mu},L_{\lambda})\cong\Ext^1_{\cO}(L_{\mu'},\theta L_{\lambda'}),$$
and a short exact sequence $L_{\lambda'}\hookrightarrow \theta L_{\lambda'}\tto Q,$ for some $s$-finite module $Q$. This yields the exact sequence
$$\Hom_{\cO}(L_{\mu'},Q)\to \Ext^1_{\cO}(L_{\mu'},L_{\lambda'})\to \Ext^1_{\cO}(L_{\mu},L_{\lambda}).$$
The first term is zero since $L_{\mu'}$ is $s$-free, so the third term is non-zero.
\hfill  $\Box$
\subsection{The left Kazhdan-Lusztig order for classical Lie superalgebras.}

In this subsection we generalise the left KL order from reductive Lie algebras to classical Lie superalgebras. We fix a classical Lie superalgebra $\fk$ with system of positive roots $\Delta^+$. Any other system of roots with the same system of even positive roots $\Delta^+_0$ leads to the same category $\cO$. In order to have a connection between the left Kazhdan-Lusztig order and the primitive spectrum, the definition can therefore not depend intrinsically on $\Delta^+$ (with the assumption that $\Delta_0^+$ remains fixed).  Since our definition will only depend on the modules, and not essentially on their highest weights (which depend on $\Delta^+$) this condition is satisfied.

Before introducing the order we need the following definition. For a simple reflection $s\in W$, we consider the corresponding positive root $\gamma$, simple in $\Delta_0^+$. The simple module $L_\lambda$ is either $X$-free or locally $X$-finite for a non-zero $X\in \fk_{-\gamma}$. In the first case $L_\lambda$ is called $s$-free, in the second $s$-finite.
\bd
\label{Hades}
The partial quasi-order $\KL$ on $P_0$ is transitively generated by the following relation. If for $\lambda,\mu\in P_0$ and a simple reflection $s\in W$
\begin{itemize}
\item[\rm (i)] $L_\lambda$ is $s$-finite and $L_\mu$ is $s$-free;
\item[\rm (ii)] $\Ext^1_{\cO}(L_\lambda, L_\mu)\not=0$;
\end{itemize}
are satisfied, we set $\mu \KL \lambda$.
\ed

\begin{proposition}
\label{Hephaestus}
If for $\lambda,\mu\in P_0$, we have $\mu\KL\lambda$, then $J_\mu(\fk)\subseteq J_\lambda(\fk)$.
\end{proposition}
\begin{proof}
This is a reformulation of Lemma 5.17 in \cite{CoMa}.
\end{proof}

\begin{proposition}
\label{intfinite}Consider a basic classical Lie superalgebra $\fg$.
\begin{itemize} 
\item[\rm (i)] The quasi-order $\KL$ is interval finite;
\item[\rm (ii)] The inclusion order is interval finite.
\end{itemize}
\end{proposition}
\begin{proof}
Theorem 5.12(ii) in \cite{CoMa} implies that property (i) would follow if we can prove that the smallest quasi-order $\KL'$ such that $\mu\KL'\lambda$ if $[T_s L_\mu : L_\lambda]\not=0$ for any simple reflection $s$, is interval finite. In other words, the consecutive procedure of taking a simple subquotient of the action of a twisting functor on a simple module should only yield a finite number of non-isomorphic modules. This is certainly true for Lie algebras, as twisting functors preserve central character (Proposition~5.11 in \cite{CoMa}). Moreover as the twisting functors are right exact (Lemma 5.4 in \cite{CoMa}) and intertwine the restriction functor (Lemma 5.1 in \cite{CoMa}), the restriction to the Lie algebra of all modules generated by the twisting functors must be composed of a finite number of simple modules for the underlying Lie algebra. That this only allows a finite number of simple modules for the Lie superalgebra follows e.g. from Lemma B.2 of \cite{CoSe}.

To prove part (ii) consider all $\mu$ for which $J_\mu\subset J_\lambda$ for a fixed $\lambda$. Take any simple subquotient of the $\fg_{\oa}$-module ${\rm Res}^{\fg}_{\fg_{\oa}} L_\lambda$. Corollary 4.2 in \cite{CoMa} implies that each ${\rm Res}^{\fg}_{\fg_{\oa}} L_\mu$ must contain one of the finitely many non-isomorphic simple $\fg_{\oa}$-modules with the same central character. This allows only a finite number of $\fg$-modules, see again Lemma B.2 of \cite{CoSe}.
\end{proof}

\subsection{The left Kazhdan-Lusztig order for $\mathfrak{gl}(m|n)$.}
In this subsection we return to $\fg=\mathfrak{gl}(m|n)$ with $\Delta^+$ as in Section \ref{secprel}. Lemma 2.1 in \cite{CoMa} implies that in this case Definition \ref{Hades} can be reformulated as follows.

\bd \label{DefKLo} The partial quasi-order $\KL$ on $\Z^{m|n}$ is transitively generated by the following relation. If for $\alpha,\beta\in\Z^{m|n}$ and a simple reflection $s\in W\cong S_m\times S_n$
\begin{itemize}
\item[\rm (i)] $s\alpha < \alpha$ and $s\beta\ge\beta$;
\item[\rm (ii)] $\Ext^1_{\cO}(L(\alpha), L(\beta))\not=0$;
\end{itemize}
are satisfied, we set $\beta \KL \alpha$.
\ed

Condition (ii) is known in principle and determined by Brundan's Kazhdan-Lusztig polynomials, see \cite{Br, BLW, CLW}. As in \cite{BLW}, see also the proof of Theorem \ref{thmKujCR}, we denote the monomial basis of the $U_q(\mathfrak{sl}(\infty))$-module $\dot{V}^{\otimes m}\otimes \dot{W}^{\otimes m}$ by $\{\dot{v}_\alpha\,|\,\alpha\in \Z^{m|n}\}$ and Lusztig's canonical basis by $\{\dot{b}_\beta\,|\,\beta\in \Z^{m|n}\}$. We define the KL polynomials by
$$\dot{b}_\beta=\sum_{\alpha\in\Z^{m|n}}d_{\alpha,\beta}(q)\dot{v}_{\alpha}\quad\mbox{and}\quad\dot{v}_{\alpha}=\sum_{\beta\in\Z^{m|n}}p_{\alpha,\beta}(-q)\dot{b}_\beta.$$
By the characterisation of Lusztig's canonical basis (see \cite{Br, BLW}) we know that $d_{\alpha,\alpha}=1$ and if $\alpha\not=\beta$ we have $d_{\alpha,\beta}\in q\Z[q]$ and $d_{\alpha,\beta}=0$ unless $\alpha \ge \beta$.
According to equation (5.29) in \cite{BLW} we have
\begin{eqnarray}\label{Helena}\dim\Ext^1_{\cO}(L(\alpha), L(\beta))&=&\left(\frac{\partial}{\partial q} p_{\alpha,\beta}\right)_{q=0}+\left(\frac{\partial}{\partial q} p_{\beta,\alpha}\right)_{q=0}.
\end{eqnarray}
This implies $\left(\frac{\partial}{\partial q} p_{\beta,\alpha}\right)_{q=0}=\left(\frac{\partial}{\partial q} d_{\alpha,\beta}\right)_{q=0}$.
We can thus define a $\mu$-function given by
$$\mu(\alpha,\beta)=\dim\Ext^1_{\cO}(L(\alpha),L(\beta))=\left(\frac{\partial}{\partial q} d_{\alpha,\beta}\right)_{q=0}+\left(\frac{\partial}{\partial q} d_{\beta,\alpha}\right)_{q=0}.$$

Concretely we proved that condition~(ii) is equivalent to
\begin{itemize}
\item[\rm (ii')] $\left(\frac{\partial}{\partial q} d_{\alpha,\beta}\right)_{q=0}\not=0$ or $\left(\frac{\partial}{\partial q} d_{\beta,\alpha}\right)_{q=0}\not=0$.
\end{itemize}
According to equation (3.1) in \cite{CoSe} we have
\begin{equation}\label{extVerma}\dim \Ext^1_{\cO}(L(\alpha), L(\beta))=\dim \Ext^1_{\cO}(M(\alpha),L(\beta))+\dim\Ext^1_{\cO}(M(\beta),L(\alpha)).\end{equation} Only one of the terms on the right-hand side can be non-zero, as for arbitrary highest weight categories.

\subsection{Discussion of the conjectural description of $\Prim_{\Z}U$ for $\mathfrak{gl}(m|n)$.}
The following conjecture is based on Theorem \ref{reformVogan}.
\begin{conjecture}
\label{thecon}
For $\fg=\mathfrak{gl}(m|n)$ and any $\alpha,\beta\in \Z^{m|n}$, we have
$$J(\beta)\subseteq J(\alpha)\quad\Leftrightarrow\quad \beta\KL \alpha.$$
\end{conjecture}
The evidence for this conjecture is summarised in the following Theorem.
\begin{theorem}
\label{summary}
Consider $\fg=\mathfrak{gl}(m|n)$ and $\alpha,\beta\in\Z^{m|n}$.
\begin{enumerate}
\item[\rm(i)] $ \beta\KL \alpha \quad\Rightarrow\quad J(\beta)\subseteq J(\alpha).$
\item[\rm(ii)] $ J(\beta)= J(\alpha)\quad\Leftrightarrow\quad \beta\KL \alpha\mbox{ and } \alpha \KL \beta.$
\item[\rm(iii)] If $\alpha,\beta$ are in the same $W$-orbit, then $\quad J(\beta)\subseteq J(\alpha)\quad\Leftrightarrow\quad \beta\KL \alpha.$
\item[\rm(iv)] If $\alpha$ or $\beta$ is typical, then $\quad J(\beta)\subseteq J(\alpha)\quad\Leftrightarrow\quad \beta\KL \alpha.$
\item[\rm(v)] If $\varepsilon_i(\alpha)=\varepsilon_i(\beta)>0$ and $\phi_i(\alpha)=\phi_i(\beta)$ for $i\in \Z$ we have $\beta \KL \alpha \,\Leftrightarrow\, \tilde{e}_i\beta \,\KL \,\tilde{e}_i \alpha.$
\item[\rm (vi)] Conjecture \ref{thecon} is true for singly atypical blocks and for $\fg=\mathfrak{gl}(2|2).$
\item[\rm(vii)] If $\alpha$ and $\beta$ are generic, then $\quad J(\beta)\subseteq J(\alpha)\quad\Leftrightarrow\quad \beta\KL \alpha.$
\end{enumerate}
\end{theorem}

Statement (ii) implies that the quasi-order $\KL$ introduces an actual partial order on the set of primitive ideals $\Prim U$ (for integral weights). Statement (v) shows the conjecture is consistent with Theorem \ref{thm1}.

The remainder of this subsection is devoted to the proof of this theorem, apart from part (vi), which will be proved in Corollary \ref{conjm1} and Corollary \ref{Hermes}. Note that the conjecture for $\mathfrak{gl}(2|1)$ follows immediately from the explicit calculation of the Kazhdan-Lusztig polynomials in Section 9.5 of \cite{CW} and the description of the primitive spectrum in Section 3 of \cite{M3}.

First we remark that (iv) is immediate from Theorem \ref{reformVogan} since the KL theory of typical blocks is the same as for the underlying Lie algebra, while (i) is a special case of Proposition \ref{Hephaestus}. Property (vii) follows immediately from (iii) as the results in \cite{CoMa} imply that an inclusion between two generic weights (as defined in Definition 7.1 of \cite{CoMa}) implies that they are in the same Weyl group orbit.

Now we find another expression for the extensions between simple modules. The first claim also follows as a special case of Lemma 3.8 in \cite{CoSe}.
\begin{lemma}
\label{lem2case}
If $\lambda,\mu\in\fh^\ast$ are in the same $\rho$-shifted (or equivalently $\rho_0$-shifted) orbit of $W$, we have
$$\dim\Ext^1_{\cO}(L_\lambda,L_\mu)=\dim \Ext^1_{\cO}(L_{\lambda}(\fg_0),L_\mu(\fg_0)).$$
If $\lambda,\mu\in\fh^\ast$ are in different orbits of $W$, we have (with $\fn_0=\fg_0\cap\fn$)
$$\dim\Ext^1_{\cO}(L_\lambda,L_\mu)=\dim\Hom_{\fh}(\C_\lambda, \left(H^1(\fg_1, L_\mu)\right)^{\fn_0})+\dim\Hom_{\fh}(\C_\mu, \left(H^1(\fg_1, L_\lambda)\right)^{\fn_0}).$$
\end{lemma}
\begin{proof}
By equations \eqref{Vermancohom} and \eqref{extVerma}, we find
\begin{equation}\label{extext}\dim \Ext^1_{\cO}(L_\lambda, L_\mu)=\dim \Hom_{\fh}(\C_\lambda,H^1(\fn,L_\mu))+\dim \Hom_{\fh}(\C_\mu,H^1(\fn,L_\lambda)).\end{equation}
Since $\fg_1$ is an ideal in $\fn$,
and  $L_\lambda^{\fg_1}\cong L_\lambda(\fg_0)$, the five term exact sequence
arising from the Hochschild-Serre spectral sequence in Example~7.5.3 in \cite{We} begins with
\begin{equation}\label{Hera}0\to H^1(\fn_0, L_\lambda(\fg_0))\to H^1(\fn,L_\lambda)\to \left(H^1(\fg_1, L_\lambda)\right)^{\fn_0}\to H^2(\fn_0, L_\lambda(\fg_0)).\end{equation}
We also have the same exact sequence with $\lambda$ replaced by $\mu$. Since all $\fh$-modules appearing above are semisimple, applying the functor $\Hom_{\fh}(\C_\mu,-)$ (respectively  $\Hom_{\fh}(\C_\lambda,-)$) to the exact sequences also yields exact sequences.

First we assume that $\lambda$ and $\mu$ are not in the same orbit. Applying $\Hom_{\fh}(\C_\mu,-)$ to the first and fourth term in \eqref{Hera} gives zero, based on equation \eqref{Vermancohom} and the central character for $\fg_0$, so we find
$$\Hom_{\fh}(\C_\mu, H^1(\fn,L_\lambda))\cong \Hom_{\fh}(\C_\mu,\left(H^1(\fg_1, L_\lambda)\right)^{\fn_0}).$$
The same reasoning with roles of $\lambda$ and $\mu$ reversed yields the result.

Now if $\lambda$ and $\mu$ are in the same orbit, we know that applying  $\Hom_{\fh}(\C_\mu,-)$ to the third term in \eqref{Hera} gives zero, since it yields a subset of $\Hom_{\fh}(\C_\mu,\fg_{-1}\otimes L(\lambda))=0$. So we find
$$\Hom_{\fh}(\C_\mu, H^1(\fn,L_\lambda))\cong \Hom_{\fh}(\C_\mu,H^1(\fn_0, L_\lambda(\fg_0)))$$
and by applying the analogue of \eqref{extext} for $\fg_0$ we obtain the claim.
\end{proof}

Using the Lemma we can prove the following consistency of the conjecture.
\begin{lemma}
For any $\alpha,\beta\in \Z^{m|n}$, we have
$$ J(\beta)= J(\alpha)\quad\Leftrightarrow\quad \beta\KL \alpha\mbox{ and } \alpha \KL \beta.$$
\end{lemma}
\begin{proof}
One direction is immediate from Proposition \ref{Hephaestus}. Now assume we have $J(\beta)=J(\alpha)$. By Theorem \ref{thmCoMa} we have $I(\beta)=I(\alpha)$ and in particular $\alpha$ and $\beta$ are in the same orbit. The result therefore follows from the combination of Theorem~\ref{reformVogan} and Lemma \ref{lem2case}.
\end{proof}

Similarly, Lemma \ref{lem2case} and Theorem \ref{thmCoMa} lead to the following result.
\begin{lemma}
If $\alpha,\beta\in \Z^{m|n}$ are in the same orbit of $S_m\times S_n\cong W$, then
$$J(\beta)\subseteq J(\alpha)\quad\Leftrightarrow \quad \beta \KL \alpha.$$
\end{lemma}

\begin{lemma}
Consider $\alpha,\beta \in \Z^{m|n}$ with $\varepsilon_i(\alpha)=\varepsilon_i(\beta)$ and $\phi_i(\alpha)=\phi_i(\beta)$ for some $i\in \Z$. If $\varepsilon_i(\alpha)>0$ (respectively $\phi_i(\alpha)>0$ ) we have $$\beta \KL \alpha \quad\Leftrightarrow\quad \tilde{e}_i\beta \,\KL \,\tilde{e}_i \alpha \qquad (\mbox{respectively }\quad\beta \KL \alpha \quad\Leftrightarrow\quad \tilde{f}_i\beta \,\KL \,\tilde{f}_i \alpha ).$$
\end{lemma}

\begin{proof}
Since $\beta\KL\alpha$, there is a finite number $p$ such that we have elements of $\Z^{m|n}$ denoted by $\{\alpha_i\,|\, i=1,\cdots,p\}$ for which
$$\beta= \alpha_p \KL \alpha_{p-1}\KL\cdots\KL\alpha_1\KL \alpha,$$
and where each two consecutive weights are related by the generating relation of $\KL$.
Proposition \ref{Hephaestus} implies that we have
$$J(\beta)=J(\alpha_p)\subseteq J(\alpha_{p-1})\subseteq\cdots\subseteq J(\alpha_1)\subseteq J(\alpha).$$
Corollary \ref{vuelta}, then implies that we have $\varepsilon_i(\alpha_k)=\varepsilon_i(\alpha)=\varepsilon_i(\beta)$ for each $1\le k\le p$.

The above paragraph thus implies that it suffices to prove the following claim. If $\alpha,\beta$ satisfy condition (i) and (ii) in Definition \ref{DefKLo} for some simple reflection $s$, and the properties concerning their signatures in the statement of the result, the weights $\widetilde e_i \alpha$ and $\widetilde e_i \beta$ satisfy condition (i) and (ii) in Definition \ref{DefKLo} for the same simple reflection~$s$.

We will use the (right exact) twisting functor $T_s$ as defined and studied in \cite{CMW, CoMa}. Theorem 5.12(ii) of \cite{CoMa} implies that
\be \label{ltwist} \dim \Hom_{\cO}(L(\alpha),T_s L(\beta))\not=0.\ee
Now consider $\varepsilon_i(\alpha)>0$.
Since the twisting functor commutes with the exact translation functor (Lemma 5.9 in \cite{CoMa}) and the functors $f_i$ and $e_i$ are adjoint to one another we find
$$\Hom_{\cO}(L(\tilde e_i\alpha), T_s e_i L(\beta))\cong \Hom_{\cO}(f_iL(\tilde e_i \alpha), T_sL(\beta)).$$
By Theorem \ref{thmKujCR}, $f_iL(\tilde e_i \alpha)$ has simple top $L(\alpha)$, implying by \eqref{ltwist}, that the above space has dimension greater than zero. This means that there must be some simple subquotient $L(\beta')$ of $e_i L(\beta)$ such that $$\dim\Hom_{\cO}(L(\tilde e_i\alpha), T_s  L(\beta'))\not=0.$$ Lemma 5.15 in \cite{CoMa} then implies that
$$J(\beta')\subseteq J(\tilde e_i\alpha).$$
Now if $\beta'\not=\tilde e_i\beta$, Theorem \ref{thmKujCR} implies that $\varepsilon_i(\beta')<\varepsilon_i(\beta)-1$, while $\varepsilon_i(\tilde e_i\alpha)=\varepsilon_i(\alpha)-1$, leading to a contradiction by Lemma \ref{bat}. This means we have $$\dim\Hom_{\cO}(L(\tilde e_i\alpha), T_s  L(\tilde e_i \beta))\not=0,$$ which through Theorem 5.12 in \cite{CoMa} implies that $s \tilde e_i\alpha < \tilde e_i\alpha$ and $s \tilde e_i \beta\ge \tilde e_i \beta$. So we find that $\tilde e_i\beta\KL\tilde{e}_i\alpha$ by Theorem~5.12 (ii) in \cite{CoMa}.
The same procedure for $\phi_i$ and $\tilde{f}_i$ concludes the proof.
\end{proof}

\subsection{The right order and the classical formulation.}
\label{sec55}
In this section we demonstrate how the more classical formulation of the inclusion order for Lie algebras fails for superalgebras. This classical order uses the right KL order $\preceq^{(r)}_{KL}$ on the Weyl group, which we can define by $x\preceq^{(r)}_{KL} y$ if and only if $x^{-1}\preceq_{KL}^{(l)} y^{-1}$. It follows from Subsection \ref{secKLorder} that this order can be described in terms of projective functors on the principal block of category $\cO$, see \cite{BG} for definition and classification. Consider $\fk$ a reductive Lie algebra. The quasi-order $\preceq^{(r)}_{KL}$ on the Weyl group corresponds to the smallest quasi-order such that $y\preceq_{KL}^{(r)} x$ if there is a projective functor $T$ on the principal block, with
$$[T L_{x\cdot 0}(\fk): L_{y\cdot 0}(\fk)]\not=0.$$
Then we have $I_{x\cdot 0}(\fk)\subseteq I_{y\cdot 0}(\fk)$ if and only if $y^{-1}\preceq^{(r)}_{KL} x^{-1}$, see \cite{J2, Jo, MaMi, Vo}, or the proof of Theorem \ref{reformVogan}.

We could introduce an analogue of the right Kazhdan-Lusztig order, by directly extending the approach via projective functors.  The use of the bijection on the set of weights, given by the inversion on the Weyl group, prevents a canonical formulation of the potential analogue of the link of this right order with the primitive spectrum for a Lie superalgebra~$\fg$. Even so, we argue that any reasonable formulation of the above principle will not give the inclusion preorder. 
For clarity, we fix an arbitrary bijection $\xi$ on the set of integral weights corresponding to a central character. The inclusion order $\KL_{\xi}$ is then defined as the smallest quasi-order such that $\mu \KL_{\xi} \lambda$ if there is a projective functor $T$ on the corresponding block, with
$$[T L_{\xi(\mu)}(\fg): L_{\xi(\lambda)}(\fg)]\not=0.$$
 Note that by the above $\KL_{\xi}$ can be seen as a different attempt to generalise the left KL order.

We focus on an example for $\fg=\mathfrak{gl}(2|1)$. For $k\ge 2$ we consider the finite dimensional simple module $L(k, 1|k)$. The functors $T:=\widetilde{f}_k\widetilde{e}_k$ and $\widetilde{T}=\widetilde{f}_{k+1}\widetilde{e}_{k+1}$ are projective functor on the block corresponding to that module. It follows easily that
\begin{equation}\label{TT}[T L(k, 1|k): L(k+1, 1|k+1 )]\not=0.\quad\mbox{and}\quad[\widetilde{T} L(k+1, 1|{k+1}): L(k, 1|k)]\not=0.\end{equation}

As finite dimensional simple modules correspond to integral dominant weights, they are fixed points in the above bijection for Lie algebras. Moreover, also for Lie superalgebras these are modules which are categorically characterised within category $\cO$, see Corollary 6.2 in \cite{CoSe} and which have annihilator ideals separated from the others by Gelfand-Kirillov dimension. It thus seems plausible that they are preserved (as a set) under $\xi$. However for any algebra, a primitive ideal with finite codimension is the annihilator of a unique finite dimensional simple module. So equation \eqref{TT}, with the assumption from earlier in this paragraph, predicts incorrect inclusions. This reasoning extends readily to any $\mathfrak{gl}(m|n)$.

Assume we do not demand the plausible condition that $\xi$ preserves finite dimensional modules. For $k>>0$ the weights are generic, see Definition~7.1 in~\cite{CoMa}. As it would be impossible for $\xi$ to map all these generic weights to non-generic ones, the above principle shows that $\KL_\xi$ would predict more inclusions (and even equalities) for Lie superalgebras in the generic region than there are for Lie algebras. This is not true, see e.g. Lemma 7.5 or Theorem~10.1 in~\cite{CoMa}. Also this extends easily to arbitrary $\mathfrak{gl}(m|n)$.

For the specific case of $\mathfrak{gl}(2|1)$, we note that for any bijection~$\xi$, the preorder $\KL_{\xi}$ cannot be the inclusion preorder, by the following immediate observations. Equation~\eqref{TT} predicts equalities between annihilator ideals for strictly different integral simple highest weight modules of $\mathfrak{gl}(2|1)$. There are no such inclusions by Theorem~\ref{thmCoMa}~(i) and the classification for~$\mathfrak{gl}(2)$.

Finally, it is immediate that the right order will in general not be interval finite, already for singly atypical blocks. This is inherited by any order $\KL_\xi$, which shows that $\KL_\xi$ can not be the inclusion order by Proposition \ref{intfinite}.


\section{Singly atypical characters and low-dimensional cases.}

\subsection{The primitive spectrum for singly atypical characters.}
\label{subsecsingatyp}

In this section we algorithmically classify all inclusions between primitive ideals for singly atypical characters for $\mathfrak{gl}(m|n)$ for integral weights. As a consequence of the proof we obtain a confirmation of Conjecture~\ref{thecon} for those blocks.

First we need to introduce some notation. We denote the unique number in $\alpha\in\Z^{m|n}$ which appears on both sides of the separator by $a_\alpha$. We also use $\pi:\{1,\cdots,m+n\}\to\{0,1\}$ with $\pi(i)=0$ iff $i\le m$.

For $\alpha\in\Z^{m|n}$ we introduce ordered sets $$\cI=\{i_{-1},i_0,i_1,\cdots,i_k\}\in [1,m+n]^{\oplus k+2}$$ for $k\ge 0$ which satisfy the following properties:
\begin{enumerate}[(i)]
\item $\alpha_{i_{-1}}=a=\alpha_{i_0}$ and $\pi(i_{-1})+\pi({i_0})=1$;
\item $\alpha_{i_j}=a_\alpha+j$ if $j>0$;
\item if $\pi(i_j)=\pi(i_l)=0$ with $-1\le j<l\le k$, then $i_j > i_l$;
\item if $\pi(i_j)=\pi(i_l)=1$ with $-1\le j<l\le k$, then $i_j < i_l$.
\end{enumerate}

We denote the largest $k$ for which we have such a set by $p_\alpha$. Note that we can always interchange the first two elements of an $\cI$ to obtain a different ordered set satisfying (i)-(iv). From now on, if $p_\ga>0$ we only consider sets where this freedom is restrained by demanding $\pi(i_0) =1-\pi(i_1)$.

The unique such ordered set $\cI$ with $|\cI|=p_\alpha+2$ in which every $i_j$ with $\pi(i_j)=0$ is chosen to be maximal and every $i_j$ with $\pi(i_j)=1$ is chosen to be minimal is denoted by $\cI_\alpha$.

For $i\in \cI_\alpha$, except the first element, we denote by $q_i$ the number of consecutive $l\in\cI_\alpha$ immediately to the right of $i$ which all satisfy $\pi(l)=1-\pi(i)$. If $i$ is the first element of $\cI_\alpha$ we set $q_i=0$. In particular we have $\sum_{j\in\cI_\alpha}q_j=p_\alpha$.

\noi We consider the example for $\mathfrak{gl}(8|4)$ where
\begin{equation}\label{examp}\alpha=(7,6,2,3,6,1,3,1|4,3,4,5),\, \mbox{so }\, \cI_\alpha=\{10,7,11,12,5,1\}\mbox{ and}\,\, p_\alpha=4.\end{equation}
Furthermore we have $q_{10}=0,$ $q_{7}=2,q_{11}=0,q_{12}=2,q_5=q_1=0$.
We will use this $\alpha$ throughout this section to illustrate certain procedures.
\\ \\
For any $p\in \Z$ and singly atypical $\alpha\in \Z^{m|n}$ we define $\Theta_\alpha^p$ as the $W$-orbit through
\be \label{Theta}(\alpha_1,\cdots,\alpha_{l-1},a_\alpha+p,\alpha_{l+1},\cdots,\alpha_m|\alpha_{m+1},\cdots,\alpha_{k-1},a_\alpha+p,\alpha_{k+1},\cdots,\alpha_{m+n}),\ee
for any $l,k$ with $\alpha_l=a_\alpha=\alpha_k$. Note that $\chi_\alpha=\chi_\beta$ if and only if $\beta\in\Theta_\alpha^p$ for some $p\in\Z$. Our main result in this section is the following theorem.

\begin{theorem}
\label{mainsingatyp}
Consider $\alpha,\beta\in\Z^{m|n}$ singly atypical. We have an inclusion $J(\beta)\subset J(\alpha)$ if and only if the following two conditions are satisfied:
\begin{itemize}
\item[\rm (i)] There is a $p\in\N$, such that $0\le p\le p_\alpha$ and $\beta\in\Theta^p_\alpha$.
\item[\rm (ii)] The inclusion $I({\delta})\subset I({\gamma})$ holds for $\mathfrak{gl}(m)\oplus\mathfrak{gl}(n)$,
\end{itemize}
with $\gamma,\delta\in\Z^{m|n}$ defined as
$$\gamma_j=\begin{cases}\alpha_j-1& \mbox{if }\alpha_j\le a_\alpha+p \quad \mbox{and }\,\,j\not\in \cI_\alpha;\\
\min(\alpha_j+q_j,a_\alpha+p) &\mbox{if }\alpha_j\le a_\alpha+p \quad \mbox{and }\,\,j\in \cI_\alpha;\\
\alpha_j &\mbox{otherwise};\end{cases}$$
$$\delta_j=\begin{cases}\beta_j-1& \mbox{if }\beta_j\le a_\beta= a_\alpha+p \, \mbox{ with  $\beta_j$ not one of the two occurrences}\\
&\mbox{of $a_\beta$ closest to the separator};\\
\beta_j &\mbox{otherwise}.\end{cases}$$
Furthermore $\gamma$ and $\delta$ are in the same $S_m$-orbit and we have $\tau(\gamma)=\tau(\alpha)$ and $\tau(\delta)=\tau(\beta)$. For $\alpha,\beta$ satisfying $(i)$ we have $J(\beta)\prec J(\alpha)\Leftrightarrow I(\delta)\prec I(\gamma)$.
\end{theorem}
Explicit examples of the algorithm will be given in Subsection \ref{secm1}.
\begin{corollary}
If $\beta\in\Z^{m|n}$ is regular, than $J(\beta)\subset J(\alpha)$ implies that $\alpha$ and $\beta$ are in the same orbit.
\end{corollary}
\begin{proof}
Any $\beta\in\Theta_\alpha^p$ with $0<p\le p_\alpha$ (and $\alpha$ arbitrary) is singular.
\end{proof}
For $\alpha$ in equation \eqref{examp} and $p\in[0,4]$, $\gamma$ is given by $\alpha^{[p]}$ in equations \eqref{alpha02} and \eqref{alpha03} below. The remainder of this section is devoted to proving Theorem \ref{mainsingatyp}. First we prove in Lemma \ref{lem1noincl} a certain condition on $\alpha\in\Z^{m|n}$ under which we can conclude that there are no inclusions $J(\beta)\subset J(\alpha)$ for any $\beta$ not in the orbit of $\alpha$ (by using Lemma~\ref{bat}). The remainder of the proof then consists of using Theorem \ref{thm1} in order to reduce to the situation where either
\begin{itemize}
\item we can use Lemma \ref{lem1noincl} to disprove possible inclusions;
\item the weights are in the same orbit, so we can use Theorem~\ref{thmCoMa} (ii) to prove or disprove possible inclusions.
\end{itemize}

\begin{lemma}
\label{lem1noincl}
Consider $\alpha\in\Z^{m|n}$ singly atypical. If there is a $\beta\in\Z^{m|n}$, not in the $W$-orbit of $\alpha$, such that $J(\beta)\subset J(\alpha)$, then $p_\alpha>0$.
\end{lemma}
\begin{proof}
Assume that $p_\alpha=0$, then either there is no label equal to $a_\alpha+1$, or there are no $a_\alpha$ in between appearances of $a_{\alpha}+1$ and the separator. In each of these scenarios all the $-$signs appear to the right of all the $+$signs in the $a_\alpha$-signature, so the reduced signature is equal to the actual signature. In other words $\varepsilon_{a_\alpha}(\alpha)$, resp. $\phi_{a_\alpha}(\alpha)$, is equal to the number of $-$signs, resp. $+$signs, in $\alpha$.

If $\beta\in\Theta_{\alpha}^p$ with $p\not\in\{0, 1\}$, then equation \eqref{Theta} implies that the $a$-signature of $\gb$ contains fewer $-$ and $+$signs than that of $\alpha$. If $p=1$, the $a$-signature of $\beta$ contains the same number of signs, but there will always be a cancellation, since there will be a $-$ sign left of the separator and a $+$ sign right of it. This contradicts Lemma~\ref{bat}. The statement follows.
\end{proof}

\bl \label{nl} Suppose $\chi_\beta=\chi_\alpha$ singly atypical, then $\beta\in\Theta^p_\alpha$ where $p\in \bbZ$.
\bi \itemi If $p\le 0$, then all labels strictly larger than $a_\alpha$ appear an equal number of times in $\alpha$ and $\beta$, and on the same sides.
\itemii
If $p\ge 0$ then all labels strictly smaller than $a_\alpha$ appear an equal number of times in $\alpha$ and $\beta$, and on the same sides. \ei \el
\bpf This follows immediately from equation \eqref{Theta}.\epf

\begin{lemma}
\label{nlnl}
Assume $\alpha\in\Z^{m|n}$ singly atypical and $\beta\in\Theta^p_\alpha$ where $p\in\Z$. Suppose $\alpha'\in\Z^{m|n}$ $($respectively $\beta')$ is obtained from $\alpha$ $($respectively $\beta)$ by raising all labels strictly bigger than $a_\alpha$ by one.  Similarly suppose $\alpha''\in\Z^{m|n}$ $($respectively $\beta'')$ is obtained from $\alpha$ $($respectively $\beta)$ by lowering all labels strictly lower than $a_\alpha$ by one.
\bi \itemi If $p\le 0$, we have $J(\beta)\subset J(\alpha)\Leftrightarrow J(\beta')\subset J(\alpha')$.
\itemii
If $p\ge 0$, we have $J(\beta)\subset J(\alpha)\Leftrightarrow J(\beta'')\subset J(\alpha'')$. \ei
\end{lemma}
By construction all weights on the right-hand side are singly atypical.
\begin{proof}
We prove (i) since (ii) is proved similarly. We use Lemma~\ref{nl}(i). Denote the numbers strictly bigger than $a_\alpha$ which appear as labels (in $\alpha$ or $\beta$) by $\{x_1,x_2,\cdots,x_k\}$ in descending order for some $k\ge 0$, and the number of times they appear respectively by $\{n_1,n_2,\cdots,n_k\}$. The case $k=0$ is trivial, so assume $k>0$. First we consider the case where the labels $x_1$ appear on the left side. Then $n_1=\phi_{x_1}(\alpha)=\phi_{x_1}(\beta)$ and $\varepsilon_{x_1}(\alpha)=0=\varepsilon_{x_1}(\beta)$, so Theorem~\ref{thm1} states that $$J(\beta)\subset J(\alpha)\quad \Leftrightarrow \quad J(\tilde f_{x_1}^{n_1}\beta)\subset J(\tilde f_{x_1}^{n_1} \alpha). $$ Set $\alpha^{(1)}:=\tilde f_{x_1}^{n_1}\alpha$ and $\beta^{(1)}:=\tilde f_{x_1}^{n_1}\beta$. If the labels equal to $x_1$ appear on the right-hand side we can do the same procedure using $\tilde e_{x_1}$.

If $k=1$ this proves the lemma. If $k>1$, by the previous step there will be no label in $\alpha^{(1)}$ or $\beta^{(1)}$ equal to $x_2+1$, so $\phi_{x_2}(\alpha^{(1)})=\phi_{x_2}(\beta^{(1)})$ and $\varepsilon_{x_2}(\alpha^{(1)})=\varepsilon_{x_2}(\beta^{(1)})$, where one of the values is $0$ and the other $n_2$. Theorem \ref{thm1} then again implies that, $J(\beta)\subset J(\alpha)$ if and only if $J(\beta^{(2)})\subset J(\alpha^{(2)}),$ where $\gamma^{(2)}$ is obtained from $\gamma^{(1)}$ by raising all entries equal to $x_2$ by one for $\gamma\in\{\alpha,\beta\}$. Iterating the procedure we eventually have $\alpha'=\alpha^{(k)}$, $\beta'=\beta^{(k)}$ and Theorem \ref{thm1} implies that $J(\beta)\subset J(\alpha)$ if and only if $J(\beta')\subset J(\alpha').$
\end{proof}
\noi The first procedure described in the Lemma applied to $\alpha$ in equation \eqref{examp} yields
\begin{eqnarray*}
\alpha^{(1)}=(8,6,2,3,6,1,3,1|4,3,4,5), && \alpha^{(2)}=(8,7,2,3,7,1,3,1|4,3,4,5),\\
\alpha^{(3)}=(8,7,2,3,7,1,3,1|4,3,4,6), && \alpha^{(4)}=(8,7,2,3,7,1,3,1|5,3,5,6)=\ga'.
\end{eqnarray*}
\noi The following is obvious from the construction, but useful for future use.
\br \label{nc} With notation as in Lemma \ref{nlnl}
\bi \itemi  $\alpha'$ and $\gb'$ do not contain any label equal to $a_\alpha+1$
\itemii $\alpha''$ and $\gb''$ do not contain any label equal to $a_\alpha-1$.\ei
\er

\begin{corollary}
\label{2ndlemm1}
Consider $\alpha,\beta\in\Z^{m|n}$ singly atypical. If $J(\beta)\subset J(\alpha)$, then $\beta\in\Theta^p_\alpha$ for $p\ge 0$.
\end{corollary}
\begin{proof}
Assume that $\beta\in\Theta^p_\alpha$ for $p<0$. By Lemma \ref{nlnl} (i), the inclusion is equivalent to $J(\beta')\subset J(\alpha')$.  However by Remark \ref{nc} (i) this  contradicts  Lemma \ref{lem1noincl}.
\end{proof}

\begin{lemma}
\label{centaur}
Consider $\gz\in\Z^{m|n}$ singly atypical and $\eta\in\Theta_{\gz}^p$ with $p>0$, $p_{\gz}>0$ and such that $a_{\zeta}$ occurs precisely once on each side of $\zeta$. Set $n$ equal to the number of times $a_\zeta+1$ appears in $\gz$ and $T=\tilde e_{a_\zeta}$ if $a_\zeta+1$ appears on the left and $T=\tilde f_{a_\zeta}$ if $a_\zeta+1$ appears on the right. We define $\hat{\gz}:=\tilde T^{n}\gz$ and $\hat{\eta}:=\tilde T^n\eta$.

Then we have $$J(\eta)\subset J(\gz)\quad\Leftrightarrow\quad J(\hat\eta)\subset J(\hat\gz),$$
where $\hat\eta\in\Theta_{\hat\gz}^{p-1}$, $p_{\hat\gz}=p_{\gz}-1$ and $a_{\hat\zeta}=a_\zeta+1$. If $\cI_\zeta=\{i_{-1},i_{0},i_{1},\cdots,i_{p_\zeta}\}$, then $$\cI_{\hat\zeta}=\begin{cases}\{i_{0},i_{1},i_2,\cdots,i_{p_\zeta}\}&\mbox{if }q_{i_0}=1 \mbox{ (equivalently $\pi(i_1)+\pi(i_2)=1$)}\\
\{i_{1},i_{0},i_{2},\cdots,i_{p_\zeta}\}&\mbox{if }q_{i_0}>1 \mbox{ (equivalently $\pi(i_0)+\pi(i_2)=1$),}\end{cases}$$
where $q_{i_0}$ refers to $\cI_\zeta$.

Furthermore $a_{\hat\zeta}$ occurs precisely once on each side of ${\hat\gz}$.
\end{lemma}
\begin{proof}
By assumption $p_\gz>0$, so there is an $a_\zeta+1$ in $\zeta$, for which there is an $a_\zeta$ between it and the separator. Assume that $a_\zeta+1$ appears on the left-hand side then there is an $l>0$ such that the $a_\zeta$-signature of $\gz$ (respectively reduced $a_\zeta$-signature) is
$$\stackrel{l}{\overbrace{--\cdots--}}+\stackrel{n-l}{\overbrace{--\cdots--}}\,|\,- \qquad \rightarrow\qquad \stackrel{n-1}{\overbrace{--\cdots--}}|-.$$
There is also an $l'\ge 0$ such that the (reduced) $a_\zeta$-signature of $\eta$ (using equation~\eqref{Theta}) is of the form
$$\stackrel{l'}{\overbrace{--\cdots--}}\,\,\,\,0/-\,\,\,\,\stackrel{n-l'}{\overbrace{--\cdots--}}\,|\, 0/+\qquad\rightarrow\qquad \stackrel{n}{\overbrace{--\cdots--}}|0,$$
where the two zeros appear if $p>1$ and the $-|+$ if $p=1$. We thus obtain
$$\varepsilon_a(\gz)=\varepsilon_a(\eta)=n\quad\phi_a(\gz)=\phi_a(\eta)=0.$$
The equivalence of inclusions is therefore implied by Theorem \ref{thm1}.

The reduced signatures of $\gz$ and $\eta$ furthermore imply that $a_{\hat{\gz}}=1+a_{{\gz}}$ and $a_{\hat{\eta}}=a_{\eta}.$ Together with $\chi_{\hat\gz}=\chi_{\hat\eta}$ (Remark \ref{Dionysos}), this implies that $\hat{\eta}\in\Theta^{p-1}_{\hat{\gz}}$. The proof when $a_\zeta+1$ appears on the right-hand side is analogous.

The fact that $a_{\hat\zeta}=a_\zeta+1$ appears precisely once on each side of ${\hat{\gz}}$ follows from construction, since ${\hat{\gz}}$ is obtained from ${{\gz}}$ by replacing all but one of the $a_\zeta+1$ by $a_\zeta$ on the side where $a_\zeta+1$ appeared, and by raising $a_\zeta$ to $a_\zeta+1$ on the side where $a_\zeta+1$ did not appear. This also proves the statement concerning $\cI_{\hat\zeta}$.
\end{proof}

\begin{proof}[Proof of Theorem \ref{mainsingatyp}]
Based on Corollary \ref{2ndlemm1} it suffices to determine when we have $J(\beta)\subset J(\alpha)$ for $\beta\in \Theta^p_\alpha$ for $p\ge 0$. We set $a:=a_\alpha$.

We use Lemma \ref{nlnl} (ii) yielding a condition $J(\beta'')\subset J(\alpha'')$ equivalent to the original inclusion.
Let $n+2$ (with $n\ge 0$) be the total number of occurrences of $a$ in $\alpha$. If $n=0$ we set $\alpha^{[0]}=\alpha''$ and $\beta^{[0]}=\beta''$. If $a$ appears more than once on the left-hand side of $\ga''$, Remark \ref{nc} (ii) implies $\varepsilon_{a-1}(\alpha'')=n=\varepsilon_{a-1}(\beta'')$ and $\phi_{a-1}(\alpha'')=0=\phi_{a-1}(\beta'')$; where the calculation for $\beta''$ depends on whether $p=0$ or $p>0$. Then we set $\alpha^{[0]}=\tilde e_{a-1}^n\alpha''$ and $\beta^{[0]}=\tilde e_{a-1}^n\beta''$. If $a$ appears more than once on the right-hand side of $\ga''$, we similarly have  $\varepsilon_{a-1}(\alpha'')=0=\varepsilon_{a-1}(\beta'')$, $\phi_{a-1}(\alpha'')=n=\phi_{a-1}(\beta'')$ and set $\alpha^{[0]}=\tilde f_{a-1}^n\alpha''$ and $\beta^{[0]}=\tilde f_{a-1}^n\beta''$. By Theorem~\ref{thm1}
$$J(\beta^{[0]})\subset J(\alpha^{[0]})\quad\Leftrightarrow \quad J(\beta)\subset J(\alpha).$$

Note that $p_{\alpha^{[0]}}=p_{\alpha}$ and $\beta^{[0]}\in\Theta^p_{{\alpha^{[0]}}}$. By construction, $\alpha^{[0]}$ contains $a$ precisely once on each side. Set $k=\min(p,p_\alpha)$. Then we can iteratively apply Lemma \ref{centaur} to obtain weights
$\alpha^{[1]}=\widehat{\alpha^{[0]}},\alpha^{[2]}=\widehat{\alpha^{[1]}},\ldots,\alpha^{[k ]}= \widehat{\alpha^{[{k-1}]}}$,
 and similarly $\beta^{[1]},\ldots,\beta^{[k]}$ for which
$$J(\beta^{[k]})\subset J(\alpha^{[k]})\quad\Leftrightarrow \quad J(\beta)\subset J(\alpha)\quad\mbox{ with } \beta^{[k]}\in\Theta^{p-k}_{\alpha^{[k]}} \mbox{ and } p_{\alpha^{[k]}}=p_\alpha-k.$$

If $p> p_\alpha$, we have $p_{\alpha^{[k]}}=0$ while $p-k>0$, so Lemma~\ref{lem1noincl} implies there is no inclusion. This proves that (i) is a necessary condition to have an inclusion.

If $p\le p_\alpha$ we have $k=p$, so $\beta^{[p]}$ and $\alpha^{[p]}$ are in the same orbit. Theorem \ref{thmCoMa} (ii) then implies that
$$I(\beta^{[p]})\subset J(\alpha^{[p]})\quad\Leftrightarrow \quad J(\beta)\subset J(\alpha).$$
We claim $\gamma=\alpha^{[p]}$ and $\delta=\beta^{[p]}$. This implies the main statement, the fact that $\gamma$ and $\delta$ are in the same $S_m\times S_n$-orbit and $\tau(\gamma)=\tau(\alpha)$,  $\tau(\delta)=\tau(\beta)$ by equation \eqref{Angliru}.

To prove this claim, we observe that $\alpha^{[0]}$ is obtained from $\alpha$ by lowering by 1 all of the labels which are lower than or equal to $a$, except the two $a$'s closest to the separator. In particular we have $\cI_{\alpha^{[0]}}=\cI_\alpha$. For $s\ge 1$, $\alpha^{[s]}$ is constructed from $\alpha^{[s-1]}$ by lowering by 1 all labels equal to $a+s$, except at the position included in $\cI_{\alpha^{[s-1]}}$, and by raising by 1 the label equal to $a+s-1$ corresponding to the second position in $\cI_{\alpha^{[s-1]}}$. The claim for $\alpha$ therefore follows from Lemma \ref{centaur}. We also have that $\beta^{[0]}$ is obtained from $\beta$ by lowering by 1 all of the labels which are lower than or equal to $a$. The procedure in Lemma \ref{centaur} shows that for $0<k<p$, $\beta^{[k]}$ is obtained from $\beta^{[k-1]}$ by lowering all labels equal to $a+k$ by one. Finally $\beta^{[p]}$ is obtained from $\beta^{[p-1]}$ by lowering by 1 all labels equal to $a+p$ except the two closest to the separator.

Finally we prove the statement concerning coverings. Suppose we have a sequence of inclusions $J(\beta)\subset J(\kappa)\subset J(\alpha)$ for some $\alpha,\beta,\kappa\in\Z^{m|n}$. By Corollary \ref{vuelta} and Theorem \ref{thm1} the procedure of the proof translates this to $J(\delta)\subset J(\kappa')\subset J(\gamma)$ for some $\kappa'\in\Z^{m|n}.$ Note that this implies that $\kappa'$ is in the orbit of $\gamma$ and $\delta$, by Corollary~\ref{2ndlemm1}. Similarly a sequence of inclusions like the latter will be translated to one like the former by applying the adjoint of the procedure. This proves the equivalence of coverings.
\end{proof}

\noi For $\alpha$ in equation \eqref{examp} we have $\alpha''=(7,6,1,3,6,0,3,0|4,3,4,5) $, and then following the proof of Theorem
\ref{mainsingatyp}
 we obtain
 \be \label{alpha02} \ga^{[0]}= \tilde e_{2}\ga'' =(7,6,1,2,6,0,3,0|4,3,4,5),\ee  and successively $\ga^{[1]} = \tilde f_{3}^2\ga^{[0]}$, $\ga^{[2]} = \tilde f_{4}\ga^{[1]}$, $\ga^{[3]} = \tilde e_{5}^2\ga^{[2]}$, $\ga^{[4]} = \tilde e_{6}\ga^{[3]},$ where
\begin{eqnarray}\label{alpha03}
 \alpha^{[1]}=(7,6,1,2,6,0,4,0|3,3,4,5)&&\alpha^{[2]}=(7,6,1,2,6,0,5,0|3,3,4,5)\\
 \alpha^{[3]}=(7,5,1,2,6,0,5,0|3,3,4,6) &&\alpha^{[4]}=(7,5,1,2,6,0,5,0|3,3,4,7).\nn
\end{eqnarray}

\begin{corollary}
\label{conjm1}
Conjecture \ref{thecon} holds for singly atypical central characters of $\mathfrak{gl}(m|n)$.
\end{corollary}
\begin{proof}
From the proof of Theorem \ref{mainsingatyp} it follows that the quasi-order $\KL'$ on $\Z^{m|n}$ defined as the inclusion order, that is
$$\beta\KL'\alpha\quad\Leftrightarrow\quad J(\beta)\subseteq J(\alpha),$$
is completely determined by the condition $\beta\KL'\alpha\,\Rightarrow \chi_\beta=\chi_\alpha$, Theorem \ref{thmCoMa} (ii), Lemma \ref{bat} and Theorem~\ref{thm1}; as these properties are the only input for the proof. By Theorem \ref{summary} (i), (iii) and (v) the quasi-order $\KL$ satisfies these properties. This implies that $\KL'$ and $\KL$ must coincide.
\end{proof}

\subsection{The primitive spectrum for $\mathfrak{gl}(m|1)$ and examples.}
\label{secm1}
All characters for $\mathfrak{gl}(m|1)$ are typical or singly atypical. In this subsection we focus on the atypical ones, simplify Theorem \ref{mainsingatyp} for the case $\mathfrak{gl}(m|1)$ and provide examples. For this case we write $\alpha=(\underline\alpha|\alpha_{m+1})$. First we note a connection between $p_\alpha$ and $d_\alpha$ as defined in equation \eqref{ddef}.

The second entry of $\cI_\alpha$ is always $m+1$, so we omit it and define $\cI_\alpha^0\in[1,m]^{p_\alpha+1}$ as the resulting ordered set. This set has an important connection to the concept of odd reflections. Recall the sequence \eqref{distm}. The module $L(\ga)$ has a unique highest weight $\gl_i$ with respect to $\mathfrak{b}^{(i)}$ and we have $\gl_i = \gl_{i-1}$  if and only if $i \in \cI^0_\alpha$.
Furthermore we can arrange
 that the highest weight vectors for the $\mathfrak{b}^{(i)}$ satisfy $v_i= v_{i-1}$ if $i\in \cI^0_\alpha$ and $v_i= e_{-\ga_i}v_{i-1}$ otherwise.

As a consequence we obtain the following lemma.
\begin{lemma}
\label{Poseidon}
For any $\alpha\in\Z^{m|1}$ with $\lambda_\alpha\in\fh^\ast$ such that $\alpha^{\lambda_\alpha}=\alpha$ we have
\bi
\itemi $d_\alpha+p_\alpha=m-1$;
\itemii $ d_\alpha + \gl_\alpha(h)= \gl_\alpha^{\ad}(h)$.
\ei
\el
\begin{proof}
In the procedure of odd reflections we also have $e_{-\alpha_i}v_{i-1}=0$ if $i\in \cI^0_\alpha$. The first statement therefore follows from the fact that $\fg_{-1}$ is supercommutative while $\{e_{-\alpha_i},i=1,\cdots,m\}$ span $\fg_{-1}$.

Part (ii) then follows immediately from the reasoning before the lemma.
\end{proof}

The combination of this with Lemma 11.6 in \cite{CoMa} yields an alternative proof of the necessary condition (i) in the following theorem.

\begin{theorem}
\label{mainm1}
Consider arbitrary $\alpha,\beta\in\Z^{m|1}$ atypical. We have an inclusion $J(\beta)\subset J(\alpha)$ if and only if the following conditions are satisfied:
\begin{itemize}
\item[\rm (i)] There is a $p\in\N$, such that $0\le p\le p_\alpha$ and $\beta\in\Theta^p_\alpha$.
\item[\rm (ii)] The inclusion $I(\underline{\delta})\subset I(\underline{\gamma})$ holds for $\mathfrak{gl}(m)$,
\end{itemize}
with $\underline\gamma,\underline\delta\in\Z^{m}$ defined as
$$\gamma_j=\begin{cases}\alpha_j-1& \mbox{if }\alpha_j\le \alpha_{m+1}+p \quad \mbox{and }\,\,j\not\in \cI^0_\alpha\\
\alpha_j &\mbox{otherwise};\end{cases}$$
$$\delta_j=\begin{cases}\beta_j-1& \mbox{if }\beta_j\le \alpha_{m+1}+p \, \mbox{ with  $\beta_j$ not rightmost occurrence in $\underline\beta$ of $\alpha_{m+1}+p$}\\
\beta_j &\mbox{otherwise}.\end{cases}$$
Furthermore $\underline\gamma$ and $\underline\delta$ are in the same $S_m$-orbit, with $\tau(\underline\gamma)=\tau(\alpha)$ and $\tau(\underline\delta)=\tau(\beta)$. Assuming $J(\beta)\subset J(\alpha)$, we have $J(\beta)\prec J(\alpha)\Leftrightarrow I(\underline{\delta})\prec I(\underline{\gamma})$.
\end{theorem}

We give three applications of to illustrate the algorithm in the theorem. The last two will also be used in Section \ref{secaug}.
\bexa
We choose $m=4$, set $\alpha=(2312|2)$ and determine all inclusions $J(\beta)\subset J(\alpha)$ for $\beta$ not in the orbit of $\alpha$. Since $p_\alpha=1$, Theorem \ref{mainm1}(i) implies $\beta\in\Theta^1_\alpha$. An exhaustive list of these weights is given by
$$(3321|3),(3231|3),\underline{(2331|3)},(3312|3),(3213|3),\underline{(2313|3)},$$
$$(3132|3),(3123|3),\underline{(2133|3)},(1332|3),(1323|3),\underline{(1233|3)}.$$
The corresponding $\underline{\delta}$ are respectively given by
$$(2310),(2130),\underline{(1230)},(2301),(2103),\underline{(1203)},$$
$$(2031),(2013),\underline{(1023)},(0231),(0213),\underline{(0123)}.$$

Since $\underline{\gamma}=(1302)$, we need to check which of the above weights corresponds to an inclusion into $I(1302)$ for $\mathfrak{gl}(4)$. The Hasse diagram for the poset of primitive ideals with regular integral central character is given   in Example 15.3.36 of \cite{M}. This reveals that only the ideals $I(0123)$, and  $I(1230)=I(1203)=I(1023)$ are contained in $I(1302)$. This implies that an exhaustive list of inclusions in $J(2312|2)$, not in the same orbit, is given by
$$J(1233|3)\mbox{ and }J(2331|3)=J(2313|3)= J(2133|3).$$
\eexa

\bexa
Consider $\alpha$ strictly dominant $(\alpha_1>\alpha_2>\cdots>\alpha_m)$ and $\beta\in\Theta^p_\alpha$ for $0\le p \le p_\alpha$, then $\underline{\gamma}$ in Theorem \ref{mainm1} is given by
$$\gamma_j=\begin{cases}\alpha_j-1& \mbox{if }\alpha_j< \alpha_{m+1}\\
\alpha_j &\mbox{otherwise};\end{cases}$$
since by regularity each of the values $\alpha_{m+1}+i$ $($with $0\le i \le p)$ appears only once in $\underline{\alpha}$. Therefore $\underline\gamma$ is a $($strictly$)$ dominant $\mathfrak{gl}(m)$-weight, thus the condition $I(\underline{\delta})\subset I(\underline{\gamma})$ becomes trivial $($since $\underline\gamma$ and $\underline\delta$ are in the same orbit$)$. This leads to the conclusion that
$$J(\beta)\subset J(\alpha)\quad\Leftrightarrow\quad \beta\in\Theta^p_\alpha\quad\mbox{with}\quad 0\le p\le p_\alpha. $$
As an extreme case we can take an $\alpha$ satisfying $\alpha_i=\alpha_{m+1}+m-i$, then we have
\begin{equation}\label{Xi} J(\beta)\subset J(\alpha)\quad\Leftrightarrow\quad \beta\in\Theta^p_\alpha\quad\mbox{with}\quad 0\le p\le m-1. \end{equation}
By choosing $\alpha_{m+1}$ correctly, this particular $J(\alpha)$ is the augmentation ideal $\fg U(\fg)$.
\eexa

\bexa
\label{exam3}
Consider $\beta\in\Z^{m|1}$ antidominant and atypical. Then $J(\beta)\subset J(\alpha)$ if and only if $\beta\in\Theta^p_\alpha$ with $p_\alpha\ge p\ge 0$. This follows immediately since $\underline{\delta}$ is also antidominant.
\eexa

\subsection{The primitive spectrum for $\mathfrak{gl}(2|2)$.}
\label{sec22}
Up to equivalence, only the principal block of $\mathfrak{gl}(2|2)$ is not singly atypical or typical. The techniques for the singly atypical cases do not lead to a classification of all inclusions for this block, see Remark \ref{Zeus} below. However,  we can obtain a complete classification by adding the result in Theorem \ref{summary} (i). As an extra result this will prove that Conjecture \ref{thecon} is true for $\mathfrak{gl}(2|2)$.

\begin{lemma}
\label{Hestia}
For $\fg=\mathfrak{gl}(2|2)$ we have
$$(11|11)\KL(10|01),\quad (21|21)\KL(10|01)\mbox{ and}\quad (12|12)\KL(10|01).$$
\end{lemma}
\begin{proof}
This follows from calculating $\dot{b}_{(10|01)}$, which implies
\by  \dim\Ext^{1}_{\cO}(L(10|01),L(11|11))&=&\dim\Ext^{1}_{\cO}(L(10|01),L(21|21))\nn \\
&=&\dim\Ext^{1}_{\cO}(L(10|01),L(12|12))=1,\nn\ey
by equation \eqref{Helena}.
\end{proof}

We determine the primitive ideals that are contained in $J(\ga)$ when $\ga\in\Z^{2|2}$ is in the $W$-orbit of $(ab|ab)$.  Since the combinatorics is not affected by adding multiples of $(11|11)$ to $\ga$ we assume that $b = 0$ and $a\ge 0$. Inclusions in one orbit are determined by Theorem \ref{thmCoMa} (ii), so we focus on the other inclusions.
\bt
\label{Aphrodite}
Suppose that
$\ga \in W(a0|a0)$ and that $\gb$ is not in the $W$-orbit of $\ga$. Then $J(\gb) \subset J(\ga)$ iff  $\ga =(10|01)$ and
\be \label{bee} \gb = (11|11),\;(21|21),\;
 (12|12), \mbox{ or }  (12|21).\ee
\et
\begin{proof}
Note that since $\gb$ is doubly atypical it must have the same labels on the left as on the right.

First suppose that $a\ge2$. Then the reduced 0 and $a$-signatures of $\ga$ are both equal to  $+-$. If  $J(\gb) \subset J(\ga)$ and 0 is not on the left of $\gb$, Lemma \ref{bat} implies that 1 is on the right (so also on the left), but this produces a $-+$ pair which cancels, so does not contribute to the reduced 0-signature of $\gb$.  Thus 0 appears as a label on both sides of $\gb$ and similarly so does $a.$ Thus $\gb$ is in the $W$-orbit of $\ga$.

Consider $a = 0$. Since the 0-signature of $\ga = (00|00)$ is equal to $++--$, Lemma~\ref{bat} implies that there is no primitive ideal strictly contained in $J(\ga).$

It remains to consider $a=1$. First consider $\alpha$ any element in the orbit except $(10|01)$. Then the reduced 1-signature of $\ga$ is $+-$ and the reduced 0-signature $+-$. Lemma \ref{bat} implies that both $1$ and $0$ must appear on both sides of $\beta$.

If $\ga=(10|01)$ and $\gb$ is as in \eqref{bee}, there is an inclusion
$J(\gb) \subset J(\ga)$ by Lemma~\ref{Hestia}, Theorem \ref{summary} (ii) and the inclusion $J(12|21)\subset J(21|21)$ which follows from Theorem \ref{thmCoMa} (ii). Finally we prove that the list \eqref{bee} is exhaustive. The reduced 1-signature of $\alpha$ is $+-$, so if $J(\beta)\subset J(\alpha)$, $\beta$ must contain a $1$ on both sides and thus $\beta\in W(c1|c1)$ for some $c\in\Z$. We have to prove that such an inclusion cannot exist if $c\not\in[0,2]$ or if $\beta=(21|12)$. The last one is excluded by Lemma \ref{bat} as it has empty reduced 1-signature. In the other cases we have $\varepsilon_1(\beta)=1=\phi_1(\beta)$, so we can apply Theorem \ref{thm1}, which states that $J(\beta)\subset J(\alpha)$ is equivalent to $J(\tilde e_1 \beta)\subset J(\tilde e_1\alpha)$. Since $\alpha'=\tilde e_1\alpha=(10|02)$,   and $\beta'=\tilde e_1 \beta\in W(c1|c2)$ are singly atypical we can apply Theorem \ref{mainsingatyp} for $\beta'\in\Theta_{\alpha'}^c$ with $c\not\in[0,2]$ while $p_{\alpha'}=2$. This proves there is no inclusion, which concludes the proof.
\epf
\begin{corollary}
\label{Hermes}
Conjecture \ref{thecon} is true for $\mathfrak{g}=\mathfrak{gl}(2|2)$.
\end{corollary}
\begin{proof}
By Corollary \ref{conjm1} we only need to prove this for the principal block. One direction of the conjecture is implied by Theorem 5.4 (i). The result then follows from Lemma \ref{Hestia}, Theorem \ref{Aphrodite} and Theorem \ref{summary} (iii).
\end{proof}

\br\label{Zeus}
The fact that Theorem~\ref{thmCoMa} (ii), Lemma \ref{bat} and Theorem \ref{thm1} suffice to classify all inclusions for degree of atypicality at most $1$, does not extend to higher degree of atypicality. For example, consider $\alpha=(10|01)$ and $\beta=(11|11)$, not in the same orbit, which satisfy
\begin{eqnarray*}
\varepsilon_{1}(\alpha)=1, \,\,\phi_{1}(\alpha)=1&&\varepsilon_{1}(\beta)=2, \,\,\phi_{1}(\beta)=2\\
 \varepsilon_{0}(\alpha)=0, \,\,\phi_{0}(\alpha)=0
 &&\varepsilon_{0}(\beta)=0, \,\,\phi_{0}(\beta)=0\\
 \varepsilon_{-1}(\alpha)=0, \,\,\phi_{-1}(\alpha)=0&&\varepsilon_{-1}(\beta)=0, \,\,\phi_{-1}(\beta)=0.
\end{eqnarray*}
The inclusion $J(\beta)\subset J(\alpha)$ can not be derived from Theorem \ref{thmCoMa} (ii) and Theorem~\ref{thm1}.
\er

We can compare the poset structure on $\Prim_{\Z}U$ for $\mathfrak{gl}(2|2)$ and $\mathfrak{gl}(2)\oplus\mathfrak{gl}(2)$, using the identification of sets given by $J(\lambda)\leftrightarrow I(\lambda)$, which is justified by Theorem~\ref{thmCoMa}~(i) and first proved in \cite{L5}. From Theorem~\ref{thmCoMa}~(ii) we know that any inclusion for $\mathfrak{gl}(2)\oplus\mathfrak{gl}(2)$ is inherited by $\mathfrak{gl}(2|2)$. Theorem \ref{Aphrodite} then implies that all `extra' inclusions for $\mathfrak{gl}(2|2)$ occur in the connected component of the poset containing the augmentation ideal $\mathfrak{g}U(\mathfrak{g})$. Note that this is an infinite connected component, based on the remarks before Theorem \ref{Aphrodite}. Part of the Hasse diagram of this connected component is presented underneath.

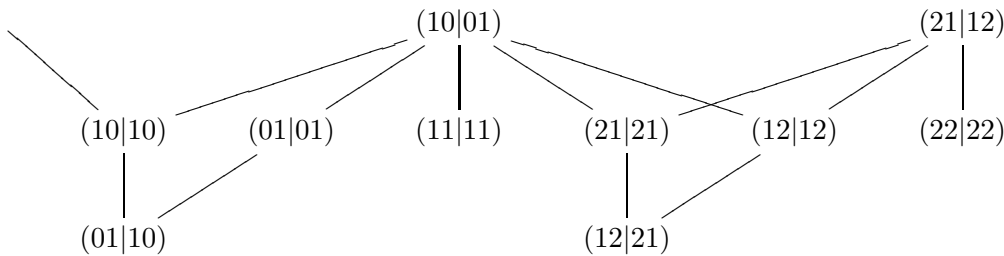
\begin{displaymath}
    \xymatrix{
&&&(10|01)&&&(21|12)\\
&(10|10)\ar@{-}[urr]\ar@{-}[ul]&(01|01)\ar@{-}[ur]&(11|11)\ar@{-}[u]&(21|21)\ar@{-}[ul]\ar@{-}[urr]&(12|12)\ar@{-}[ull]\ar@{-}[ur]&(22|22)\ar@{-}[u]\\
&(01|10)\ar@{-}[u]\ar@{-}[ur]&&&(12|21)\ar@{-}[ur]\ar@{-}[u]&&
   }
\end{displaymath}

By interpreting this diagram we also find that all irreducible components (see Theorem \ref{thmcomp}) of the topological space $\Prim_{\Z}U(\mathfrak{gl}(2|2))$ are isomorphic, as a poset, to some irreducible component of  $\Prim_{\Z}U(\mathfrak{gl}(2)\oplus \mathfrak{gl}(2))$, except for $Z(k-1,k|k,k-1)$, which contains two maximal elements.

\br
The fact that we obtain an infinite connected component would no longer hold if we would consider $\mathfrak{sl}(2|2)$. This is not a general feature however. This infinite connected component for $\mathfrak{gl}(2|2)$ leads to an infinite connected component for any $\mathfrak{gl}(m|n)$ if $m\ge 2$ and $n\ge 2$ by parabolic induction, see Corollary 4.7 in \cite{CoMa}. This would still lead to an infinite connected component when looking at $\mathfrak{sl}(m|n)$, if either $m>2$ or $n>2$.
\er


\section{Primitive ideals contained in the augmentation ideal for $\mathfrak{gl}(m|1)$.}
\label{secaug}
The ideal $J_0=\fg U(\fg)$, known as the augmentation ideal of $U(\fg)$, is the annihilator of the trivial module $L_0\cong \C$.  For $\fg=\mathfrak{gl}(m|1)$ we define the poset and topological space
\[ X=\{J \in \Prim U(\fg)| J \subseteq J_0\}.\]
Similarly, let ${\mathscr{X}}\subset\Prim U(\fg_0)$ be the poset of primitive ideals contained in the augmentation ideal of $\fg_0$.

The motivation to study the specific example $X$, in the depth we will, is threefold:
\begin{itemize}
\item It provides a good setting to study the behaviour of irreducible components in $\Prim U$  for the Jacobson-Zariski topology, see Theorem \ref{thmcomp}. We find that all the irreducible components are isomorphic, as posets, to ${\mathscr{X}}$.
\item A tool which can be complementary to the machinery developed in this paper is the application of different systems of positive roots (linked together by odd reflections), see e.g. the star actions in \cite{CoMa}. The poset $X$ provides an excellent test case, leading to two stratifications corresponding to the distinguished and anti-distinguished system of positive roots.  These stratifications also have interesting relations to the irreducible components. For $\mathfrak{gl}(m|1)$ with $m < 6$, we prove that every inclusion in $X$ can be derived from star actions, through Corollary 8.4 in \cite{CoMa}.
\item Although the description of the poset $\Prim U(\mathfrak{gl}(m|1))$ by the validity of Conjecture \ref{thecon} is very satisfactory from a conceptual point of view, and the one in Theorem \ref{mainm1} is very useful to quickly check inclusions, we seek more insight into the poset structure of $\Prim U(\mathfrak{gl}(m|1))$. It seems that the subposet $X$ is the right candidate to focus on, as it displays all new phenomena. The stratification mentioned in the previous item provides a way to see the connected components of the poset as built out of posets isomorphic to ${\mathscr{X}}$.
\end{itemize}

Contrary to the corresponding poset for $\mathfrak{gl}(2|2)$, we will find that $X$ is also the connected component of the poset $\Prim U$ containing the augmentation ideal.

\subsection{The poset $X$.}
We introduce some notation. For $0\le i\le m-1$, set $$\gl_i := \gep_{m-i+1}+\ldots+ \gep_{m}-i\gd\,\in\, \fh^*$$
and, if $i \ge 1$, $\gc_i :=\gep_{m-i} -\gep_{m+1-i}$.  Let $s_i$ be the reflection corresponding to $\gc_i.$
Denote the dot orbit of $\gl_i$ by $\Gt_i$ and set $X_i = \{J_\mu| \mu \in\Gt_i\}\subset \Prim U$. Note that $\gl_i$ is in the closure of the dominant Weyl chamber, and its stabiliser under the dot action is~$s_i$ if $i>0$. Since the set
$\{w\in W|\gc_i\in \gt(w)\}$ is the set of longest coset representatives for $(s_i)$ in $W$, we have
\be \label{sit}{X_i} = \{J_{w \cdot \gl_i}|\gc_i\in \gt(w)\}\quad\mbox{for $i>0$}.\ee
Note that the convention on $\tau$-invariants for singular weights in combination with the choice of longest coset representatives yields $\tau(w\cdot\lambda_j)=\tau(w)$.

For $i>0$ we also define a subposet of $\mathscr{X}$ as
\be \label{sat}\mathscr{X}_i = \{I_{w \cdot 0}|\gc_i\in \gt(w)\}
\subset {\mathscr{X}}
.\ee
\bt \label{sun1}
We have the disjoint union $X = \bigcup_{i=0}^{m-1} X_i$ as sets.
 \et
\begin{proof}
This is precisely equation \eqref{Xi}, where the disjointness is implied by Theorem~\ref{thmCoMa}.
\end{proof}
The subposets $X_i$ of $X$ are described by the following theorem.

\bt \label{gnu}  There are isomorphisms of posets
\[ X_0\lra {\mathscr{X}} ,\quad  J_{w\cdot 0} \lra I_{w\cdot 0} \]
and for $i>0$
\[X_i \lra {\mathscr{X}}_i,\quad  J_{w\cdot \gl_i} \lra I_{w\cdot 0} \quad \mbox{ if } \gc_i\in \tau(w).\]
\et

\bpf The first statement follows from Theorem \ref{thmCoMa} (ii). For the second, we use the parallel descriptions of the posets \eqref{sit} and \eqref{sat}. Then the statement follows from Theorem \ref{thmCoMa} (ii) and Theorem \ref{tppi} (iii).
\epf

\begin{theorem}\label{Tolstoy}
The poset $X$ is the connected component of $\Prim U$ that contains the augmentation ideal. Consequently, the closed subsets of the topological space $X$ are precisely the subsets of $X$ which are closed in $\Prim U$.
\end{theorem}
Before proving this we prove the following lemma.
\begin{lemma}
\label{tequila}
The poset $X$ contains $m-1$ minimal elements, given by $Q_i:=J_{w_0\cdot\lambda_i }$ for $0\le i\le m-1$.
\end{lemma}
\begin{proof}
By Theorem \ref{sun1}, each poset $X_i$ has a minimal element $J_{w_0\cdot\lambda_i }$. By Lemma~\ref{joueur} there are no inclusions between these ideals, meaning that each $J_{w_0\cdot\lambda_i }$ is actually minimal in $X$.
\end{proof}

\begin{proof}[Proof of Theorem \ref{Tolstoy}]
By construction $X$ is connected, it thus suffices to prove that it is maximal. By Lemma \ref{tequila}, a necessary and sufficient condition to prove that $X$ is a connected component of $\Prim U(\fg)$ is thus
$$J_{w_0\cdot\lambda_i}\subset J_\kappa\quad \Rightarrow\quad J_\kappa\in X\qquad\mbox{ for all }0\le i\le m-1.$$
To prove this we will work with the notation of Section \ref{secm1}, so we have
\be \label{ark}
\alpha^{\lambda_j}=\begin{cases}(m-1,m-2,\cdots,1,0|0)& \mbox{for $j=0$},\\
(m-1,m-2,\cdots,j+1,j,j,j-1,\cdots,1,1|j)&\mbox{for }1\le j\le
m-1.\end{cases}\ee
We define $\beta:=\alpha^{w_0\cdot\lambda_k}$. Then we have by \eqref{ark},
$$\beta=(1,2,\cdots, k-1,k,k,k+1,\cdots,m-1|k).$$

According to Corollary \ref{2ndlemm1}, an inclusion $J(\beta)\subset J(\alpha)$ implies that $\alpha\in\Z^{m|1}$ is in the orbit of
$$(1,2,\cdots, k-1,k-t,k,k+1,\cdots,m-1|k-t)\qquad\mbox{with}\quad t\ge 0.$$
Example \ref{exam3} implies that in order to have $J(\alpha)\not\in X$, we need $t>k$.

Since then $k-t<0$ and there is no label equal to $0$, all such $\alpha$ have $p_\alpha=0$.  But $ t> k\ge 0$ then contradicts Theorem \ref{mainm1}.
\end{proof}

\noi
We end this subsection with a technical lemma concerning the $\tau$-invariants (as defined in Section \ref{secprel}) of elements of $\Theta_j$.
\begin{lemma}
\label{Elpis}
Suppose $\lambda\in\Theta_j$ for $0\le j\le m-1$, and set $\ga= \alpha^\lambda\in\Z^{m|1}.$ Then
\bi
\itemi If $j=0$, we have $\gamma_k\in\tau(\lambda)\Leftrightarrow $ $k$ appears
to the right of $k-1$ in $\underline{\alpha}$.
\itemii If $j>0$ then $\gamma_j\in\tau(\lambda)$ 
 and for $k\neq j$, $\gamma_k\in\tau(\lambda)$ if and only if one of the following holds
\bi \itema $\,\,\;k<j$  \mbox{and $k+1$ appears to the right of $k$ in } $\underline{\alpha}$
\itemb $\,\,\; k=j+1$ \mbox{and $j+1$ appears to the right of both of the $j$'s in } $\underline{\alpha}$
\itemc $\,\,\;  k>j+1$ \mbox{and $k$ appears to the right of $k-1$ in } $\underline{\alpha}$.
\ei \ei
\end{lemma}
\begin{proof}
The only non-trivial case is where $k = j+1$ for $j>0$. The reason that $j+1$ needs to be to the right of both of the $j$'s corresponds to our chosen convention where $\gamma_j\in\tau(w\cdot\lambda_j)$ for all $w\in W$.
\end{proof}

\subsection{A double stratification.}
Consider the stratification of $X$ in Theorem \ref{sun1}. The term
stratification will be justified in Theorem \ref{jig}. This states
that $\bigcup_{i=0}^s X_i$ is a closed subspace of $X$ for all $0\le s\le
m-1$, which implies that Theorem \ref{sun1} provides a filtration of $X$ by
closed subspaces. The antidistinguished system of positive roots leads in a similar fashion to another stratification of $X$. In this subsection we study the link between both stratifications. The expression for $\rho$ formed using the distinguished system of positive roots is given in equation \eqref{rho}.  Using the antidistinguished system we have

\[\gr^\ad=\frac{1}{2}\sum_{i=1}^m (m+2-2i)\gep_{i} -\frac{1}{2} m\gd.\]
Clearly the $\rho^{\ad}$-shifted action of the Weyl group corresponds to the $\rho$-shifted (and thus the $\rho_0$-shifted) action, so there is no need to specify which dot action is used.

Now for $0\le j\le m-1$ we set
$$\mu_j := -\gep_{1}-\ldots -\gep_{j}+j\gd,$$ and denote the dot orbit of $\mu_j$ by $\Phi_j$ and
$Y_j = \{J_\mu| \mu^\ad \in \Phi_j\}$.
Note that $\mu_j$ is also in the closure of the dominant Weyl chamber and if $j>0$ its stabiliser under the dot action is $s_{m-j}$.
By symmetry, Theorem \ref{sun1} extends to the following.
\bt \label{sun}
We have disjoint unions
\be \label{bc} X = \bigcup_{i=0}^{m-1} X_i = \bigcup_{i=0}^{m-1} Y_i.\ee
 \et

Now we investigate the connection between both stratifications. The main result is stated in the following theorem, for which we introduce the notation $\Theta=\bigcup_{i=0}^{m-1}\Theta_i\subset P_0$ and $\Phi=\bigcup_{i=0}^{m-1}\Phi_i\subset P_0$. We also use the convention $\max\emptyset=0$. Recall the definition of $p_\alpha$ for $\alpha\in\Z^{m|n}$ from Subsection \ref{subsecsingatyp}, which we extend to $p_\lambda:=p_{\alpha^\lambda}$ for any $\lambda\in P_0$.

\begin{theorem}\label{XYconn}
For $\lambda\in \Theta_i$, that is $\lambda=w\cdot\lambda_i$ for some $w\in W$ $($where we assume $\gamma_i\in \tau(w)$ if $i>0)$, we have
$$\lambda^{ad}=w\cdot \mu_j\qquad \mbox{with}\quad j\,\,\,=\,\,\,\max\{ k<m-i\,|\, \gamma_{m-k}\in\tau(w)\}\,\,\,=\,\,\,m-i-1-p_\lambda.$$
\end{theorem}
In particular this demonstrates how the minimal elements $Q_i$ of $X$ behave with respect to the double stratification, i.e. in which $Y_j$ the unique minimal element of $X_i$ plays the role of unique minimal element.
\begin{corollary} \label{prob}$ \;$ The minimal element $Q_i:=J_{w_0\cdot \gl_i}$ is contained in $X_i$ and $Y_{m-i-1}.$\end{corollary} \bpf
To know which $Y_j$ the ideal $Q_i$ belongs to we need to calculate $(w_0\cdot \gl_i)^\ad.$ Since $p_{w_0\cdot\lambda_i}=0$, Theorem \ref{XYconn} gives $(w_0\cdot \gl_i)^\ad=w_0\cdot \mu_{m-i-1}$. \epf
\noi To state another immediate consequence, for $J_\gl  \in X_i \cap Y_j$, we set $i(\gl) =i,\; j(\gl)=j.$
\bc
If $i(\lambda)=i(\mu)$ and $\tau(\mu)=\tau(\lambda)$, then $j(\lambda)=j(\mu)$.
\ec
\noi The remainder of this subsection is devoted to the proof of Theorem \ref{XYconn}. Recall $h\in\mathfrak{z}(\fg_0)$ introduced in Section \ref{secprel}.
\bl \label{1.8}
We have \bi \itemi  $i(\gl) =-\gl(h), \;j(\gl)=\gl^\ad(h);$
\itemii
 $i(\gl) +j(\gl)= d_\gl \le m-1;$
\itemiii If $J_\mu  \subseteq J_\gl $ then $j(\mu) \ge j(\gl)$ and  $i(\mu) \ge i(\gl).$
\ei \el
\bpf The first property follows since it  holds for $\gl_i$ and $\mu_j$, and $h$ is $W$-invariant. Property (ii) follows from (i) and Lemma \ref{Poseidon} (ii).
Property (iii) follows from Lemma~11.6 in \cite{CoMa} or alternatively Corollary~\ref{2ndlemm1}.
\epf

We will need the following general technical lemma.
\begin{lemma}
\label{medusa}
Take $\kappa\in\fh^\ast$ regular or such that there are unique $1\le i_0<j_0\le m$ such that $\langle \kappa+\rho,\epsilon_{i_0}-\epsilon_{j_0}\rangle=0$. There is a $w\in W$ such that both $w^{-1}\cdot\kappa$ and $w^{-1}\cdot\kappa^{\ad}$ are dominant.
\end{lemma}
\begin{proof}
We consider the case where $\kappa$ is singular, since the proof for regular $\kappa$ corresponds to a simplified version of the proof we give below.

There is a $u\in W$ such that $u^{-1}\cdot\kappa$ is dominant. Then there is a unique $1\le t <m$ such that $\epsilon_{t}-\epsilon_{t+1}=\pm u^{-1}(\epsilon_{i_0}-\epsilon_{j_0})$ and we let $s_0\in W$ be the simple reflection corresponding to this simple root. Then $s_0u^{-1}\cdot \kappa=u^{-1}\cdot\kappa$ is also dominant.
From the procedure for odd reflections it follows that for $1\le i \le m$, either the coefficients of $\epsilon_i$ in $\kappa$ and $\kappa^{\ad}$ are equal, or
the coefficient of $\epsilon_i$ in $\kappa$ is one more than
the corresponding coefficient in  $\kappa^{\ad}$. Therefore we have for any root $\gamma\in \Delta_0$,
$$\langle \kappa+\rho_0,\gamma\rangle >0\quad\Rightarrow\quad \langle \kappa^{\ad}+\rho_0,\gamma\rangle \ge 0.$$
This implies that for any $i$ excluding $t$ we have
$$\langle u^{-1}\cdot \kappa^{\ad} +\rho_0, \epsilon_i-\epsilon_{i+1}\rangle =\langle \kappa^{\ad}+\rho_0, u(\epsilon_i-\epsilon_{i+1})\rangle\ge 0,$$
where the same property holds for $s_0u^{-1}$. Finally, since $us_0(\epsilon_t-\epsilon_{t+1})=-u(\epsilon_t-\epsilon_{t+1})$, we have
$$\langle u^{-1}\cdot \kappa^{\ad} +\rho_0, \epsilon_t-\epsilon_{t+1}\rangle\,\,=\,\,- \,\langle s_0u^{-1}\cdot \kappa^{\ad} +\rho_0, \epsilon_t-\epsilon_{t+1}\rangle.$$
So either $u^{-1}\cdot\kappa^{\ad}$ or $s_0u^{-1}\cdot\kappa^{\ad}$ is dominant.
\end{proof}

\begin{lemma}
\label{Persephone2}
Consider $\lambda\in\Theta_i$, for $0\le i\le m-1$. We have $p_\lambda=l-i-1$ with
$$l:=\begin{cases} \quad m &\quad \mbox{if }\{k|\gamma_k\in\tau(\lambda) \mbox{ with } k>i\}=\emptyset\\
 \min\{k |\gamma_k\in\tau(\lambda) \mbox{ with } k>i\}& \quad\mbox{otherwise}.\end{cases}$$
\end{lemma}
\begin{proof}
We focus on the case $i>0$, which is the more difficult one to prove. If $i+1$ is to the
right of both of the two $i$ in the even part
of $\alpha^\lambda$ , then by definition $p_\gl=0$.
Otherwise, \[p_\lambda= 1+ \max\{r| i+s+1 \mbox{ is to the left of } i+s \mbox{ for  } 1\le s\le r\}.\]
 The result thus follows from Lemma \ref{Elpis}.
\end{proof}

\begin{proof}[Proof of Theorem \ref{XYconn}]
We prove the formulation in terms of $p_\lambda$, the approach using $\tau$-invariants then follows from Lemma \ref{Persephone2}.

By Lemmata \ref{Poseidon} and \ref{1.8} (i) we know that $\lambda^{\ad}\in\Phi_j$ for $j:=m-p_\lambda-i-1$. In case $i=0$, Lemma \ref{medusa} implies that $\lambda^{ad}=w\cdot \mu_j$. In case $i>0$, Lemma \ref{medusa} implies that either $\lambda^{\ad}=w\cdot\mu_j$ or $\lambda^{\ad}=ws_i\cdot\mu_j$. The fact that the longer element $w$ must be taken follows from the procedure of odd reflections, which shows that if $\langle\lambda+\rho_0, \epsilon_a-\epsilon_b\rangle=0$ (with $a<b$) implies that $\langle\lambda^{\ad}+\rho_0, \epsilon_a-\epsilon_b\rangle\le0$. For the particular case of $\lambda=w\cdot\lambda_j$ and $\epsilon_a-\epsilon_b=w(\gamma_j)$ one can even show that we will always have a strict inequality.
\end{proof}
\bc \label{jog} For $\gl \in \Gt $ we have $i(\gl) + j(\gl) + p_\lambda=m-1.$\ec
\bpf This follows from Lemma \ref{Poseidon} (i) and Lemma \ref{1.8} (ii).\epf

\subsection{The irreducible components of $X$.}
In this subsection we study the irreducible components in Theorem \ref{thmcomp} of $\Prim U$ given by $Z(w_0\cdot\lambda_k)$. By Theorem \ref{Tolstoy} we have
$$Z_k :=Z(w_0\cdot\lambda_k) =\{J \in X|Q_k \subseteq J\}.$$
In  combination with Lemma \ref{tequila} this implies that the $Z_k$ are precisely the irreducible components of the topological space $X$, as $X = \bigcup_{k=0}^{m-1} Z_k$.

The main results concerning $Z_k$ are presented in the following two theorems.

\begin{theorem}
\label{Hercules} We have the equivalent characterisations
$$J_\lambda\in Z_k\quad\Leftrightarrow\quad i(\lambda)\le k\le m-1-j(\lambda);$$
$$J_\lambda\in Z_k\quad\Leftrightarrow\quad i(\lambda)\le k\le i(\lambda)+p_\gl.$$
\end{theorem}
\noi The equivalence of the two statements follows from Corollary \ref{jog}.
\begin{theorem}
\label{Athena}
The poset $Z_k$ is isomorphic to $X_0$ and thus to $\mathscr{X}$.
\end{theorem}
\begin{proof}[Proof of Theorem \ref{Hercules}]
We only need to translate the result in Example \ref{exam3} to our notation. The condition $p\ge 0$ is equivalent to $i(\lambda)\le k$, the condition $p_\alpha\ge p$ translates to $p_\lambda \ge k-i(\lambda)$. Lemma \ref{Poseidon}(i) and Lemma \ref{1.8}(ii) yield $p_\lambda=m-i(\lambda)-j(\lambda)-1$, showing that the necessary and sufficient condition becomes $i(\lambda)\le k$ and $m-j(\lambda)>k$.
\end{proof}

\begin{proof}[Proof of Theorem \ref{Athena}]
We start from the description of $Z_k$ given in Theorem \ref{Hercules},
\begin{equation}\label{Demeter}Z_k\,\,\,=\,\,\,\bigcup_{i=0}^k\,\,\{J_\lambda\,|\,\lambda\in \Theta_i\mbox{ and }p_\lambda\ge k-i\}.\end{equation}

Since the case $k=0$ is trivial we focus on $k>0$. We will prove that application of $E_{k-1}E_{k-2}\cdots E_0$ (as defined in Section \ref{Apollo}) maps $Z_k$ to the sub-poset $X_0^{(k)}$ of $\Prim U$ corresponding to the $W$-orbit through $(m-1,m-2,\cdots,1,0|k)$. We know that $X_0^{(k)}$ isomorphic to $\mathscr{X}$ by Theorem~\ref{thmCoMa}~(ii).

We claim that the $0$-signatures of weights $\lambda$ corresponding to equation \eqref{Demeter} all satisfy $\varepsilon_0=1$ and $\phi_0=0$. For $\lambda\in\Theta_0$ this follows from the fact that there we must have $p_\lambda>0$, implying that the $1$ must appear to the left of the $0$ in the even part. For $\lambda\in \Theta_i$ with $i>0$ this claim is always true, without any condition. This means that $E_0$ yields an isomorphism of posets $Z_k \xra E_0(Z_k)$ (with inverse $F_0$) by Theorem \ref{thm1}. For $\lambda\in\Theta_0$, the action of  $\tilde e_0$ will raise the odd part of $\alpha^\lambda$ from $0$ to $1$. For $\lambda\in\Theta_i$ with $i>0$, the action of $\tilde e_0$ will lower the leftmost $1$ in the even part of $\alpha^\lambda$ to~$0$.

From similar arguments it follows that $E_{k-1}E_{k-2}\cdots E_0$ gives an isomorphism of posets between $Z_k$ and some poset of primitive ideals where all corresponding weights $\mu$ satisfy $\alpha^\mu_{m+1}=k$. Furthermore, since $\tilde e_{k-1}\tilde e_{k-2}\cdots \tilde e_0 \alpha^{\lambda_0}=(m-1,m-2,\cdots,1,0|k)$ and all weights for the poset possess the same central character (remark \ref{Dionysos}), the latter poset corresponds to a subposet of
$$X_0^{(k)}=\{J(w(m-1,m-2,\cdots,1,0|k))\,|\, w\in W\}.$$

Therefore it only remains to be proved that the entire poset in the equation above is reached. By similar arguments as above,  the action of $F_0F_1\cdots F_{k-1}$ yields an injective map of posets from $X_0^{(k)}$ into some subposet of $X$. Since every ideal in $X^{(k)}_0$ contains $J(0,1,\cdots,m-1|k)$ and $$\tilde f_0\tilde f_1\cdots \tilde f_{k-1}(0,1,\cdots,m-1|k)=(1,2,\cdots, k-1,k,k,k+1,\cdots,m-1|k)=\alpha^{w_0\cdot\lambda_{k}},$$
 we have $F_0F_1\cdots F_{k-1}(X_0^{(k)})\subset Z_k$. This concludes the proof and furthermore shows that $E_{k-1}E_{k-2}\cdots E_0$ and $F_0F_1\cdots F_{k-1}$,
restricted to the domains $Z_k$ and $X_0^{(k)}$ respectively, are inverse to one another.
\end{proof}
\subsection{Local Closure.}\label{subsecLC}

\noi Recall that a subset of a topological space is {\it locally closed} if it is the intersection of an open set and a closed set.
\bt \label{jig} The sets $\bigcup_{i=0}^sX_i$ and $\bigcup_{i=0}^sY_i$ are closed in $\Prim U$, whereas the sets
$X_i, Y_j$ and $X_i\cap Y_j$ are locally closed in $\Prim U$.
\et
\noi First we prove a relation between the Zariski closed sets $Z_k$ and the intersections $X_i\cap Y_j$ formed from the two stratifications.
\begin{proposition} \label{ant}
For $0\le i,j\le m-1$ we have
$$X_i\cap Y_j=\left(\bigcap_{i\le k\le m-1-j}Z_k\right)\backslash\left(\bigcap_{i-1\le k\le m-1-j}Z_k\,\,\cup\,\, \bigcap_{i\le k\le m-1-j+1}Z_k\right).$$
\end{proposition}
\begin{proof}
We start by proving the equality
$$\bigcup_{0\le s\le i\;,\;0\le t\le j}(X_s\cap Y_t)\quad=\quad\bigcap_{i\le k\le m-1-j}Z_k.$$
That the left-hand side is contained in the right-hand side follows immediately from Theorem \ref{Hercules}. Now assume that the primitive ideal $J_\lambda$ is contained in the right-hand side. If the value $i(\lambda)$ were bigger than $i$, $J_\lambda$ could not be contained in $Z_i$ by Theorem \ref{Hercules}. The same reasoning for $j(\lambda)$ proves the equation.

The result then follows from the equality between $X_i\cap Y_j$ and
$$\left( \bigcup_{0\le s\le i\,,\,0\le t\le j}(X_s\cap Y_t)\right) \,\,\,\backslash\,\,\,\left(\bigcup_{0\le s\le i-1\,,\,0\le t\le j}(X_s\cap Y_t)\,\,\,\cup\,\,\, \bigcup_{0\le s\le i\,,\,0\le t\le j-1}(X_s\cap Y_t)\right). $$
\end{proof}

\bp \label{cat} For $0\le  k\le m-1,$ set
\[X_{k,k}=X_k,\qquad X_{0,k} = \{J_\gl\in X_0| \gamma_{j}\notin\gt(\gl) \mbox{ for } 1\le j\le k\}\quad\mbox{ and}\]
\[X_{i,k} = \{J_\gl\in X_i| \gamma_{k},\gamma_{k-1},\ldots,\gamma_{i+1} \notin\gt(\gl), \gamma_{i} \in \gt(\gl)\}\quad\mbox{for } 0<i<k.\]
Then \bi \itemi We have a disjoint union $Z_k=\bigcup_{i=0}^k
X_{i,k}$.
\itemii $X_0 = Z_0$ and for $0<k<m-1,$ $X_k = Z_k\backslash (Z_{k-1}\cap Z_k)$.
\itemiii $Y_0 = Z_{m-1}$ and for $0<k<m-1,$ $Y_k = Z_{m-k-1}\backslash
(Z_{m-k}\cap Z_{m-k+1})$.
\ei
\ep
\bpf Obviously the union in (i) is disjoint since the sets $X_i$ are.  By Theorem \ref{Hercules}, if $J_\gl \in X_i$ then $J_\gl \in Z_k$ iff
$i\le k\le i +p_\gl.$ Thus (i) follows from Lemma \ref{Persephone2}.
Then since $X_{i,k} \subseteq X_{i,k-1}$ for $0\le i \le k-1,$ (ii) follows from (i), and (iii) is proved similarly. \epf

Applying parts (ii) and (iii) yields the following immediate conclusion.
\begin{corollary}
\label{Greg}
We have $$\bigcup_{i=0}^sX_i=\bigcup_{i=0}^s Z_i\quad\mbox{ and }\quad \bigcup_{i=0}^sY_i=\bigcup_{i=0}^s Z_{m-1-i}$$ for $0\le s\le m-1$. This implies in particular that $\bigcup_{i=0}^sX_i$ and $\bigcup_{i=0}^sY_i$ are closed.
\end{corollary}

\noi Note that Theorem \ref{jig} follows immediately from Propositions \ref{ant} and \ref{cat} and Corollary \ref{Greg}.
\subsection{Covering.}

Recall the double stratification (\ref{bc}). Theorem \ref{thmCoMa} (ii) applied to both the distinguished and anti-distinguished system of positive roots then yields a number of inclusions on $X$. It is an interesting question whether the minimal partial order created from those inclusions coincides with the inclusion order. This question can be reformulated as ``do exceptional coverings exist?'', using the definition below.
\bd \label{Ares} A covering  $J_\mu\prec J_\gl$, where both $i(\mu)\not=i(\lambda)$ and $j(\mu)\not=j(\lambda)$, is called exceptional.\ed
\noi When there are no exceptional inclusions, this means that all inclusions can be derived from the principle of star actions, see Corollary 8.4 in \cite{CoMa}.

For Lie algebras we can have strict inclusions between primitive ideals with the same $\tau$-invariant. This property is of course inherited by $\mathfrak{gl}(m|n)$ by Theorem \ref{thmCoMa}~(ii), for primitive ideals corresponding to one orbit. We prove that in the poset $X$ inclusions with constant $\tau$-invariant are only possible for inclusions between two primitive ideals in the same orbit.
\begin{lemma}
\label{Pontus}
The inclusion $J_\mu\subset J_\lambda$ for $\lambda,\mu\in\Theta$ with $i(\mu)>i(\lambda)$ implies that $\gamma_{i(\mu)}\not\in\tau(\lambda)$ and $$\tau(\mu)\,\,\supseteq\,\, \tau(\lambda)\,\cup\,\{\gamma_{i(\mu)}\}.$$
\end{lemma}
\begin{proof}
This is a direct application of Theorem \ref{mainm1}. We thus use the identification $P_0\leftrightarrow \Z^{m|n}$ and set $\beta:=\alpha^\mu$, $\alpha:=\alpha^\lambda$ and $i_1=i(\lambda)$, $i_2=i(\mu)$. In the notation of Theorem \ref{mainm1} we have $p=i_2-i_1$, so the inclusion $J(\beta)\subset J(\alpha)$ thus implies $p_\alpha\ge i_2-i_1$.  Lemma~\ref{Elpis} then yields
\begin{equation}\label{Elvishasleftthebuilding}\gamma_{\ell}\not\in \tau(\lambda)\mbox{ for }{i_1+1}\le \ell \le {i_2}.\end{equation}
Now by definition $\underline\gamma$ is obtained from $\underline\alpha$ by subtracting $1$ from the left of the two labels equal to $i_1$, and from all labels equal to an element in $[1,i_1-1]$. Similarly, $\underline\delta$ is obtained from $\underline\beta$ by subtracting 1 from the left of the two $i_2$ and all labels equal to an element in $[1,i_2-1]$. This immediately implies that for $k\in [1,i_1]\cup [i_2+1,m-1]$ we have
$$\gamma_k\in \tau(\lambda)\,\Leftrightarrow\, \gamma_k\in \tau(\underline\gamma)\,\Rightarrow\, \gamma_k\in\tau(\underline\delta)\,\Leftrightarrow\, \gamma_k\in \tau(\mu),$$
where the middle $\Rightarrow$ is a consequence of Theorem \ref{tppi} (ii) and the inclusion $I(\underline\gd) \subseteq I(\underline\gc)$. The statement then follows from observing that by definition $\gamma_{i_2}\in\tau(\mu)$.
\end{proof}

\begin{lemma}
\label{mystery}
Assume that $\mathfrak{gl}(m)$ satisfies the property
$$I_2\prec I_1 \quad\Rightarrow \quad \sharp \tau(I_2)\,\, \le\,\, 1+\sharp \tau(I_1), $$
for any two $I_1,I_2\in {\mathscr{X}}$, with $\sharp \tau(\cdot)$ the number of roots in the $\tau$-invariant. Then there are no exceptional inclusions for $\mathfrak{gl}(m|1)$.
\end{lemma}
\begin{proof}
Assume we have a covering in $X$ of the form $J_\mu\subset J_\lambda$ with $i(\mu)>i(\lambda)$, we need to prove that $j(\mu)=j(\lambda)$. Theorem \ref{mainm1} implies that $(\tau(\lambda),\tau(\mu))$ correspond to the set of two $\tau$-invariants corresponding to a covering between annihilator ideals for modules with highest weight in the same orbit.

 Theorem \ref{thmCoMa} (ii) and the assumption on $\mathfrak{gl}(m)$ thus yields $\sharp\tau(\mu)\le 1+\sharp \tau(\lambda)$. From Lemma \ref{Pontus} we thus obtain (with disjoint union)
\begin{equation}\label{Hurricane}\tau(\mu)\,\,=\,\, \tau(\lambda)\,\cup\,\{\gamma_{i(\mu)}\}.\end{equation}

Theorem \ref{XYconn} states that
\begin{eqnarray*}j(\lambda)&=&\max\{ k<m-i(\lambda)\,|\, \gamma_{m-k}\in\tau(\lambda)\}\\
j(\mu)&=&\max\{ k<m-i(\mu)\,|\, \gamma_{m-k}\in\tau(\mu)\}.
\end{eqnarray*}
Equation \eqref{Elvishasleftthebuilding} applied to the formula for $j(\lambda)$ and equation \eqref{Hurricane} applied to the one for $j(\mu)$ then yield
$$j(\lambda)=\max\{ k<m-i(\mu)\,|\, \gamma_{m-k}\in\tau(\lambda)\}=j(\mu),$$
which concludes the proof.
\end{proof}

As we have no proof that the assumption on $\mathfrak{gl}(m)$ is true for general $m$, we end this subsection with four results about situations where we can exclude exceptional coverings. This justifies the term exceptional covering.

\bl \label{hog}\bi \itemo
\itemi If $J_\mu\subset J_\lambda$, then $p_\lambda - p_\mu\ge i(\mu)-i(\lambda)\ge 0$.
\itemii There are no exceptional coverings if $p_\mu -p_\gl \le 1$, in notation of Definition~\ref{Ares}.
\ei \el \bpf Statement (i) follows from Lemma \ref{1.8} (ii) and (iii) and Lemma \ref{Poseidon} (i). For (ii) we assume we have an inclusion $J_\lambda\subset J_\mu$. If $p_\mu=p_\lambda$, then $i(\lambda)=i(\mu)$ by item (i). Assume that $p_\mu=p_\lambda+1$, if $i(\lambda)=i(\mu)$, they are both in $X_{i(\lambda)}$, if $i(\lambda)=i(\mu)+1$, they are in the same $Y_l$, by Corollary \ref{jog}.
\epf 
\bl \label{owl} There are no exceptional coverings (with notation of Definition \ref{Ares}) if either $i(\mu) = 0$,  $j(\mu)= 0$, $i(\lambda)=m$ or $j(\lambda)=m$.\el
\begin{proof}
This is an immediate consequence of Lemma \ref{1.8} (iii).
\end{proof}

\bp\label{goose} There are no exceptional coverings if $\lambda =0,$ that is when $J_\lambda$ is the augmentation ideal.  \ep
\bpf We need to prove that the ideals that $J_0$ covers which are not in $X_0$ are in $Y_0$. From the structure of the posets $X_i$ for $i>0$ we know that each of them has a unique maximal element, corresponding to $J_{\gl_i}$. All of these are in $Y_0$.\epf

\begin{proposition}\label{Sydney}
There are no exceptional coverings if $J_\mu=Q_k$ for $0\le k\le m-1$.
\end{proposition}
\begin{proof}
For $k=0$, the result is a special case of Lemma \ref{owl}, so we consider $k>0$. Suppose $Q_k\prec J_\lambda$ with $\lambda\in\Theta_l$ for $l<k$. We consider Theorem \ref{mainm1} with $\alpha:=\alpha^{\lambda}$ and $\beta:=\alpha^{w_0\cdot \lambda_k}$. This yields $\underline{\delta}=(0,1,\cdots,m-1)$. In order to have an inclusion we need $p_\alpha\ge k-l$, meaning that $\tau(\lambda)$ can contain at most $m-1-k+l$ elements.

In order to have a covering for $\mathfrak{gl}(m)$, with $I(\underline\delta)$, we need that $\tau(\underline{\gamma})$ contains precisely $m-2$ elements. As $\tau(\lambda)=\tau(\underline{\gamma})$ we find $k=l+1$. By Theorem \ref{mainm1} the only $Q_k\prec J_\lambda$ with $\lambda\in\Theta_{k-1}$ is
$$Q_k\prec J(1,\cdots, k-1,k,k-1,k+1,\cdots, m-1|k-1).$$
Clearly, for this case we have $j(\lambda)=j(w_0\cdot\lambda_k)=m-1-k$.
\end{proof}

\subsection{The inclusions for $\mathfrak{gl}(3|1)$.}
\label{Kivilev}
Below we give the Hasse diagram for the poset of primitive ideals $X$ when $\fg=\mathfrak{gl}(3|1)$.

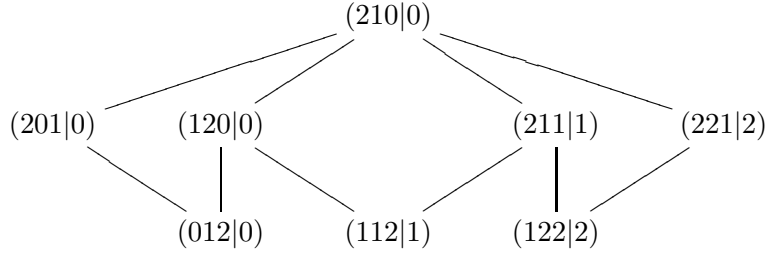
\begin{displaymath}
    \xymatrix{
&&(210|0)&&\\
(201|0)\ar@{-}[urr]&(120|0)\ar@{-}[ur]&&(211|1)\ar@{-}[ul]&(221|2)\ar@{-}[ull]\\
&(012|0)\ar@{-}[ul]\ar@{-}[u]&(112|1)\ar@{-}[ur]\ar@{-}[ul]&(122|2)\ar@{-}[ur]\ar@{-}[u]&
   }
\end{displaymath}

\noi Each ideal is labeled by $\ga^\gl$, where $\gl$ is the highest weight of the module it annihilates. 
Note that we have equalities
$$J(2,0,1|0)=J(0, 2, 1|0),\quad
J(1,2,0|0)=J(1,0,2|0),$$ $$J(2,1,1|1)=J(1,2,1|1),\quad J(2,2,1|2)=
J(2,1,2|2).$$
We describe the double stratification \eqref{bc} in terms of the diagram.  The set $X_i$ consists of all ideals whose last entry is $i$.  In particular the maximal ideals in $X_1, X_2$ have labels
\[ \ga^{\gl_1} = (2,1,1|1),\quad \ga^{\gl_2} = (2,2,1|2).\]
On the other hand $Y_1$ consists of the annihilators of the simple modules 
\[L^\ad_{\mu_1} = L(1,2,0|0), \quad  L^\ad_{w_0\cdot\mu_1} = L(1,1,2|1), \]
and $Y_2$ consists of the annihilators of the simple modules
\[L^\ad_{\mu_2} = L(2,0,1|0), \quad  L^\ad_{w_0\cdot\mu_2} = L(0,1,2|0), \]
while $Y_0$ consists of the annihilators of the four remaining modules 
\[L^\ad_{0} = L(2,1,0|0), \;  L^\ad_{s_1\cdot 0} = L(2,1,1|1), 
\; L^\ad_{s_2\cdot 0} = L(2,2,1|2),  \; L^\ad_{w_0\cdot 0} = L(1,2,2|2).\]
\subsection{The inclusions for $\mathfrak{gl}(4|1)$ and $\mathfrak{gl}(5|1)$.}

The poset $X$ for $\mathfrak{gl}(4|1)$ and $\mathfrak{gl}(5|1)$ is completely determined by the following theorem.
\bt \label{dog} There are no exceptional coverings when $m<6$.\et
\begin{proof}
That there are no exceptional coverings for $\mathfrak{gl}(2|1)$ and $\mathfrak{gl}(3|1)$ follows immediately from Subsection \ref{Kivilev} and \cite{M3}.

By Lemma \ref{mystery}, it suffices to prove that the poset corresponding to the augmentation ideal does not contain any coverings between primitive ideals of which the $\tau$-invariant differs by more than one element in $\mathfrak{gl}(m)$ for $m\in\{4,5\}$. The Hasse diagrams of these posets are presented on page 39 of \cite{BJ}. The thick lines connect the primitive ideals corresponding the same $\tau$-invariant. It can easily be checked that every descending path from top to bottom contains precisely $m-1$ vertices that are not thick. As the $\tau$-invariant of the top edge is empty and that of the bottom one contains $m-1$ roots, it follows that in each vertex that was not thick, the $\tau$-invariant must have grown precisely by one.
\end{proof}

\subsection{A generating function.} The poset $X$ studied in this section seems to be new to representation theory.
In this subsection  we determine the cardinality of $|X|$ as a function of $m$. Therefore we denote $m$ explicitly by using the notation $X^{(m)} = \bigcup_{i=0}^{m-1}X^{(m)}_i$, for the poset and its stratification in Theorem \ref{sun1}.

We set $t_m = |X^{(m)}|$ and denote the number of involutions in $S_m$ by $s_m$.  This is equal to the number of standard tableau with $m$ entries, and also the cardinality of $X^{(m)}_0$, by Theorem \ref{gnu}. There is a closed expression for $s_m$ in \cite{F} Chapter 4, Exercise 6.
 \bl For $1\le i\le m-1$ we have $|X^{(m)}_i| =|X^{(m)}_0|/2$.\el
\bpf Let  $S$ be the set of standard tableaux with $m$ entries.
For $w\in S_m$  we have, see \cite{M} Lemma 15.3.32, $\gc_i \in \gt(w^{-1})$ iff
\be \label{fat} i+1 \mbox{ is in a strictly lower row of } T \mbox{ than }i \ee where $T= B(w)\in S$, see \eqref{elf} for notation. We claim that exactly half of the elements of $S$ satisfy \eqref{fat}. Indeed, there is an involution on $S$ taking a tableau $T$ to its transpose $T^t$, which is without fixed points if $m>1$, and it is easy to see that exactly one of $T, T^t$ satisfies \eqref{fat}.\epf
\bc \label{rod} We have $t_m= (m+1)s_m/2$. \ec \bpf Immediate. \epf
\noi There are nice exponential generating functions for $s_m$ and $t_m$.

\begin{proposition}
Set $F(x) = \sum_{m=0}^\infty \frac{s_m}{m!}x^m$ and $G(x) = \sum_{m=0}^\infty \frac{t_m}{m!}x^m.$ Then
$$F(x) = \exp(x+x^2)\qquad\mbox{and}\qquad G(x) = \frac{1}{2}(1+x+2x^2)F(x) .$$
\end{proposition}
\begin{proof}
The expression for $F(x)$ is Exercise 8.19 in \cite{St}. The result for $G(x)$ then follows from a direct calculation based on Corollary \ref{rod}.
\end{proof}


\subsection*{Acknowledgments.}

The authors thank Jonathan Brundan and Volodymyr Mazorchuk for indispensable discussions.

\end{document}